\documentclass{article}
\usepackage[utf8]{inputenc}

\usepackage[margin=1.5in]{geometry}
\usepackage{amsmath,amssymb,amsthm}
\usepackage{mathtools}

\usepackage{tikz}

\usepackage{algorithm}
\usepackage{algorithmicx}
\usepackage[noend]{algpseudocode}

\usepackage[labelfont=bf]{caption}
\usepackage{graphicx}
\usepackage{dsfont}
\usepackage{xcolor}

\usepackage{newtxtext,newtxmath}

\definecolor{Base02}{HTML}{1E3668}

\definecolor{Red}{HTML}{dc322f}
\definecolor{Magenta}{HTML}{d33682}
\definecolor{Violet}{HTML}{6c71c4}
\definecolor{Cyan}{HTML}{2aa198}
\definecolor{Green}{HTML}{859900}

\usepackage[title]{appendix}

\usepackage{hyperref}
\usepackage[capitalize,nameinlink,noabbrev]{cleveref}
\crefformat{equation}{(#2#1#3)}
\crefmultiformat{equation}{(#2#1#3)}%
{ and~(#2#1#3)}{, (#2#1#3)}{ and~(#2#1#3)}

\usepackage{hyperref}
\hypersetup{
    colorlinks,
    linkcolor=Base02,
    citecolor=Base02,
    urlcolor=Base02
}

\usepackage{lipsum}
\usepackage{marginnote}
\usepackage{bm}
\usepackage{xspace}

\renewcommand{\Re}{\operatorname{Re}}
\renewcommand{\Im}{\operatorname{Im}}

\newtheorem{theorem}{Theorem}
\newtheorem{lemma}{Lemma}
\newtheorem{corollary}{Corollary}

\theoremstyle{definition}
\newtheorem{definition}{Definition}

\numberwithin{theorem}{section}
\numberwithin{lemma}{section}
\numberwithin{definition}{section}
\numberwithin{corollary}{section}

\newcommand{\cT}{*}
\renewcommand{\vec}{\mathbf}

\newcommand{\lan}{\textup{\textsf{lan}}}
\newcommand{\err}{\textup{\textsf{err}}}
\newcommand{\res}{\textup{\textsf{res}}}
\renewcommand{\d}{\mathrm{d}}
\newcommand{\F}{\mathsf{F}}

\newcommand{\R}{\mathbb{R}}

\newcommand{\fH}{f(\vec{H})}
\newcommand{\fHV}{f(\vec{H})\vec{V}}

\title{A posteriori error bounds for the block-Lanczos method for matrix function approximation}
\author{Qichen Xu\thanks{The University of Chicago, Chicago, Illinois, U.S.A, \href{mailto:qichenxu@uchicago.edu}{\texttt{qichenxu@uchicago.edu}}, corresponding author} \and
Tyler Chen\thanks{New York University, New York, New York, U.S.A, \href{mailto:tyler.chen@nyu.edu}{\texttt{tyler.chen@nyu.edu}}}}
\date{}

\begin{document}

\maketitle

\begin{abstract}
    We extend the error bounds from [SIMAX, Vol. 43, Iss. 2, pp. 787-811 (2022)] for the Lanczos method for matrix function approximation to the block algorithm. 
    Numerical experiments suggest that our bounds are fairly robust to changing block size and have the potential for use as a practical stopping criterion. 
    Further experiments work towards a better understanding of how certain hyperparameters should be chosen in order to maximize the quality of the error bounds, even in the previously studied block size one case.
\end{abstract}

\section{Introduction}
\label{sec:intro}

Lanczos-based methods have proven to be among the most effective algorithms for approximating $\fH\vec{v}$, where $f:\R\to\R$ is a scalar function, $\vec{H}$ is a $n\times n$ Hermitian matrix, and $\vec{v}$ is a vector.
As a result, a number of analyses aim to provide a posteriori error bounds which may be used as practical stopping criteria \cite{frommer_simoncini_08b,frommer_simoncini_09,ilic_turner_simpson_09,frommer_kahl_lippert_rittich_13,frommer_guttel_schweitzer_14a,frommer_schweitzer_15,chen_greenbaum_musco_musco_22a}.
In particular, for a large class of functions, \cite{chen_greenbaum_musco_musco_22a} shows a general approach to bounding the error for the approximation of the Lanczos function approximation algorithm (referred to as Lanczos-FA) to $\fH\vec{v}$ in terms of the error of Lanczos-FA used to approximate a fixed linear system $(\vec{H}-w \vec{I})\vec{x} = \vec{v}$.
In the case that $\vec{H}-w\vec{I}$ is positive definite, Lanczos-FA is mathematically equivalent to the well-known and optimal conjugate gradient algorithm \cite{greenbaum_97}, and the error of this linear system can then be accurately estimated or bounded using existing techniques. 
In fact, the residual can always be used as an estimate for the error of the approximate solution to the system $(\vec{H}-w \vec{I})\vec{x} = \vec{v}$.

Often, one would like to approximate a sequence $\fH\vec{v}_1, \ldots, \fH\vec{v}_b$ for some set of vectors $\{\vec{v}_1, \ldots, \vec{v}_b\}$.
Examples of applications in which such a situation arises include solving multiple systems of linear equations with varying right-hand sides, computing Lyapunov exponents for dynamical systems, low-rank and diagonal approximation of matrix functions, model order reduction, sampling Gaussian vectors, and studying the equilibrium thermodynamics of open quantum systems  \cite{dieci_vanvleck_95,ando_chow_saad_skolnick_12,barkouki_bentbib_jbilou_15,fika_mitrouli_roupa_16,li_marlin_16,birk_18,seki_yunoki_20,chen_hallman_23,chen_cheng_22}.
While it is possible to apply the standard Lanczos-FA to each product independently, it is often more efficient to apply a blocked version of Lanczos-FA to approximate $\fHV$ directly, where $\vec{V} = [\vec{v}_1, \ldots, \vec{v}_b]$ is a ``block-vector'' (i.e. a tall/skinny $n\times b$ matrix) \cite{golub_underwood_77,frommer_lund_szyld_17}.

In this paper, we extend the bounds of \cite{chen_greenbaum_musco_musco_22a} to the block-Lanczos-FA algorithm. 
In particular, we show that, for piece-wise analytic $f$, the error of the block-Lanczos-FA algorithm can be bounded by the product of (i) the error of the block-Lanczos-FA approximation to the block system $(\vec{H} - w\vec{I})\vec{X} = \vec{V}$ and (ii) a certain contour integral which can be approximated numerically from quantities made available by the block-Lanczos algorithm.
As in \cite{chen_greenbaum_musco_musco_22a}, the bounds depend on the choice of $w$ as well as a certain contour of integration $\Gamma$.

Since bounds and stopping criteria for block-Lanczos-FA are less studied than for standard Lanczos-FA, we believe that the present work provides a useful tool for practitioners using the block algorithm. 
We include a number of numerical experiments exploring the impact of block size on our bounds, as well as experiments that provide further intuition as to how hyper-parameters such as $w$ and $\Gamma$ should be chosen. 
These experiments are relevant for the block size one case as well and serve as a step towards addressing some of the practical questions raised in \cite{chen_greenbaum_musco_musco_22a}.

\subsection{Notation}
\label{sec:notation}

We denote vectors with bold lowercase Roman letters and matrices with bold uppercase Roman letters. Constants are denoted with unbolded lowercase Roman letters. Throughout, $b$ is the block size of a matrix, $k$ is the number of iterations of the block-Lanczos algorithm, and $\vec{E}_i$ is defined as the $(kb)\times b$ block matrix whose $i$-th $b\times b$ block is the identity matrix and zero otherwise.
We write $\lambda_{\textup{min}}$ and $\lambda_{\textup{max}}$ for the smallest and largest eigenvalues of $\vec{H}$ respectively.
The spectrum of a matrix $\vec{B}$ is denoted by $\Lambda(\vec{B})$.

We use the notation $\|\cdot\|_2$ and $\|\cdot\|_\F$ to respectively denote the two-norm and Frobenius norm of a vector or a matrix.
We will use $\| \cdot \|_{\vec{M}}$ to denote some fixed norm induced by a positive definite matrix $\vec{M}$ by the relation $\|\vec{X}\|_{\vec{M}} = \|\vec{M}^{1/2}\vec{X}\|_\F$.
In the case that $\vec{X}$ has a single column, this coincides with the standard definition of a matrix-induced norm.
In all cases, $\vec{M}$ will have the form $\vec{M} = h(\vec{H})$ for some function $h$.
Hence, for any function $g$ and matrices $\vec{X}, \vec{Y}$, such norms satisfy the inequality
\begin{align}
\label{eqn:norm_ineq}
    \|g(\vec{H}) \vec{X}\vec{Y}\|_{h(\vec{H})} 
    &= \| h(\vec{H})^{1/2} g(\vec{H}) \vec{X}\vec{Y} \|_\F
    = \| g(\vec{H}) h(\vec{H})^{1/2} \vec{X}\vec{Y} \|_\F
    \nonumber
    \\&\hspace{4em}\leq \| g(\vec{H}) \|_2 \| h(\vec{H})^{1/2} \vec{X} \|_\F \| \vec{Y} \|_2
    = \| g(\vec{H}) \|_2 \|  \vec{X} \|_{h(\vec{H})} \| \vec{Y} \|_2.
\end{align}

\subsection{The block-Lanczos method for matrix function approximation}

Given a matrix $\vec{V}\in \mathbb{R}^{n\times b}$, the degree $j$ block-Krylov subspace is defined as
\begin{equation*}
    \mathcal{K}_{j+1}(\vec{H},\vec{V}) := \operatorname{span}\{ \vec{V}, \vec{H}\vec{V}, \ldots, \vec{H}^j \vec{V} \}.
\end{equation*}
Here the span of a set of block vectors is interpreted as the span of the constituent columns.
The block-Lanczos algorithm \cite{underwood_75,golub_underwood_77}, which we described in \cref{alg:block_lanczos}, computes\footnote{For simplicity of exposition, we assume the block-Krylov subspace does not become degenerate; i.e., that the dimension of the matrix block-Krylov subspace $\mathcal{K}_{k+1}(\vec{H},\vec{V})$ is $(k+1)b$.} matrices $\{\overline{\vec{Q}}_j\}$, $j\leq k+1$ such that $\overline{\vec{Q}}_i^\cT \overline{\vec{Q}}_j = \vec{0}$ for all $i\neq j$ and $\overline{\vec{Q}}_i^\cT\overline{\vec{Q}}_i = \vec{I}$. 
Moreover, these matrices satisfy, for all $j\leq k+1$,
\begin{align*}
\operatorname{span}\{ \overline{\vec{Q}}_1, \overline{\vec{Q}}_2, \ldots, \overline{\vec{Q}}_j \}
= \mathcal{K}_{j+1}(\vec{H},\vec{V}).
\end{align*}
As a result, the block vectors $\{\overline{\vec{Q}}_j\}$ satisfy a symmetric block-tridiagonal recurrence
\begin{align}
\label{eqn:krylov_recurrence}
    \vec{H} \vec{Q}_k = \vec{Q}_k \vec{T}_k + \overline{\vec{Q}}_{k+1} \vec{B}_k \vec{E}_k^\cT
\end{align}
where 
\begin{align*}
    \vec{T}_k = \begin{bmatrix}
    \vec{A}_1 & \vec{B}_1^\cT \\
    \vec{B}_1 & \ddots & \ddots \\
    & \ddots & \ddots & \vec{B}_{k-1}^\cT \\
    &&\vec{B}_{k-1} & \vec{A}_k
    \end{bmatrix}
    ,&&
    \vec{Q}_k = \begin{bmatrix}
    | &|&  & |\\
    \overline{\vec{Q}}_1 & \overline{\vec{Q}}_2 & \cdots & \overline{\vec{Q}}_k \\
    | &|&  & |
\end{bmatrix}
\end{align*}
are block matrices of size $(bk)\times (bk)$ and $n\times (bk)$ respectively.
We also have the relation $\vec{V} = \overline{\vec{Q}}_1 \vec{B}_0$.
The $b\times b$ diagonal blocks $\{\vec{A}_i\}$ and the $b\times b$ upper triangular off-diagonal blocks $\{\vec{B}_i\}$ are also generated by the block-Lanczos algorithm.

The output of the block-Lanczos algorithm can be used to approximate $f(\vec{H}) \vec{V}$:
\begin{definition}
\label[definition]{def:block_LanczosFA}
The block-Lanczos method for matrix function approximation (block-Lanczos-FA) for $f(\vec{H}) \vec{V}$ is defined as
\begin{equation*}
    \label{eqn:fT}
    \lan_k(f) 
    := 
    \vec{Q}_k f(\vec{T}_k) \vec{E}_1 \vec{B}_0.
\end{equation*}
\end{definition}
In the case that $b=1$, the block-Lanczos-FA algorithm is the same as the well-known Lanczos-FA algorithm; see for instance \cite{saad_92}.

\begin{algorithm}
\caption{block-Lanczos algorithm}
\label{alg:block_lanczos}
\fontsize{10}{10}\selectfont
\begin{algorithmic}[1]
\Procedure{block-Lanczos}{$\vec{H}, \vec{V}, k$}
\State \( \overline{\vec{Q}}_1,\vec{B}_0 = \textsc{qr}(\vec{V}) \),
\For {\( j=1,2,\ldots,k \)}
    \State \( \vec{Z} = \vec{H} \overline{\vec{Q}}_{j} - \overline{\vec{Q}}_{j-1} \vec{B}_{j-1}^\cT \) \Comment{if $j=1$, $\vec{Z} = \vec{H} \overline{\vec{Q}}_1$}
    \State \( \vec{A}_j = \overline{\vec{Q}}_{j}^\cT \vec{Z}  \)
    \State \( \vec{Z} = \vec{Z} - \overline{\vec{Q}}_{j} \vec{A}_j \)
    \State optionally, reorthogonalize $\vec{Z}$ against $\overline{\vec{Q}}_0, \ldots, \overline{\vec{Q}}_{j-1}$
    \label{alg:reorth}
    \State \( \overline{\vec{Q}}_{j+1},\vec{B}_{j} = \textsc{qr}(\vec{Z}) \)
\EndFor
\State \Return $\{\overline{\vec{Q}}_j\}, \{\vec{A}_j\}, \{ \vec{B}_j \}$
\EndProcedure
\end{algorithmic}
\end{algorithm}

\subsection{Past work and existing error bounds}
\label{sec:past_work}

It is straightforward to show that for any polynomial $p$ with $\deg(p)<k$, \( p(\vec{H})\vec{V} = \lan_k(p) \).
That is, block-Lanczos-FA approximates low-degree polynomials exactly.\footnote{In fact, a somewhat more general statement in terms of matrix-polynomials is true \cite[Theorem 2.7]{frommer_lund_szyld_20}.}
Writing the smallest and largest eigenvalues of $\vec{H}$ as $\lambda_{\textup{min}}$ and $\lambda_{\textup{max}}$ respectively, the previous observation implies that block-Lanczos-FA satisfies the error bound:
\begin{equation}
\label{eqn:unif_bound}
    \| f(\vec{H}) \vec{V} - \lan_k(f) \|_2
    \leq 
    2 \|\vec{V}\|_2 \min_{\deg(p) < k} \left(
    \max_{x\in[\lambda_{\textup{min}},\lambda_{\textup{max}}]} | f(x) - p(x)| \right).
\end{equation}
This bound is well-known in the block size one case, and the proof is easily generalized to block sizes $b>1$ (we provide the proof in \cref{sec:appendix} for convenience).

While \cref{eqn:unif_bound} provides a certain optimality guarantee for block-Lanczos-FA, the bound often fails to capture the true behavior of block-Lanczos-FA. 
This means it is unsuitable for use as a practical stopping criterion.
At a high level, the main issue is that \cref{eqn:unif_bound} does not depend on the spectrum of $\vec{H}$ besides through $\lambda_{\textup{min}}$ and $\lambda_{\textup{max}}$ and does not depend on $\vec{V}$ except through $\|\vec{V}\|_2$.
The lack of dependence on spectral properties of $\vec{H}$ means the bound is unable to take into account properties such as clustered outlying eigenvalues which often result in accelerated convergence for Lanczos-based methods. 
A nice discussion of this phenomenon for the case of linear systems is given in \cite{carson_liesen_strakos_22} and a detailed discussion of the shortcomings of bounds like \cref{eqn:unif_bound} for Lanczos-FA can be found in \cite{chen_greenbaum_musco_musco_22a}.
The lack of dependence on $\vec{V}$ beside $\|\vec{V}\|_2$ means the bound does not depend meaningfully on the block size. However, in many (but not all) cases, the block-Lanczos-FA algorithm converges more quickly when the block size is larger. 

One case in which stronger guarantees can be easily derived is if $f(x) = (x-z)^{-1}$ and $\vec{H}-z\vec{I}$ is positive definite.
Indeed, observe that the solution to
\begin{equation*}
    \min_{\vec{X}\in \mathcal{K}_k(\vec{H},\vec{V})} \| (\vec{H}-z\vec{I})^{-1} \vec{V} - \vec{X} \|_{\vec{H}-z\vec{I}}
    = \min_{\vec{C}\in \mathbb{R}^{(bk)\times b}} \| (\vec{H}-z\vec{I})^{-1/2} \vec{V} - (\vec{H}-z\vec{I})^{1/2} \vec{Q}_k\vec{C} \|_{\F}
\end{equation*}
is obtained by solving $b$ independent least squares problems corresponding to the columns of $\vec{V}$.
This means the matrix $\vec{C}^\star$ minimizing the least squares problem on the right is
$\vec{C}^\star = (\vec{Q}_k^\cT(\vec{H}-z\vec{I})\vec{Q}_k)^{-1} \vec{Q}_k^\cT \vec{V}$.
Thus, the minimizer for the original problem can be written explicitly as
\begin{equation*}
    \vec{X}^\star
    = \vec{Q}_k\vec{C}^\star
    = \vec{Q}_k(\vec{Q}_k^\cT(\vec{H}-z\vec{I})\vec{Q}_k)^{-1} \vec{Q}_k^\cT \vec{V}
    = \vec{Q}_k (\vec{T}_k-z\vec{I})^{-1} \vec{Q}_k^\cT \vec{V}
    = \lan_k(f).
\end{equation*}
This optimality condition is well-known in the block size one case and implies that, for any block size $b$, block-Lanczos-FA is mathematically equivalent to the block-conjugate gradient algorithm \cite{oleary_80}.

The optimality of block-Lanczos-FA also implies that the convergence of any single column of the block algorithm is not slower than the single-vector variant applied to the corresponding column (at least when measured with respect to $\|\cdot\|_{\vec{H} - z\vec{I}}$).
Indeed, the block algorithm is optimal over the block-Krylov subspace $\mathcal{K}_k(\vec{H},\vec{V})$, which contains the Krylov subspace $\mathcal{K}_k(\vec{H},\vec{v}_i)$, for any column $\vec{v}_i$ of $\vec{V}$.

\begin{definition}
\label[definition]{def:err_and_res}
For all $z\in\mathbb C$ where $\vec{H}-z\vec{I}$ and $\vec{T}_k-z\vec{I}$ are invertible, we respectively define the error and residual block vector $\err_k(z)$ and $\res_k(z)$ as
\begin{align*}
    \err_k(z) &:= (\vec{H}-z\vec{I})^{-1} \vec{V} -  \vec{Q}_k (\vec{T}_k - z\vec{I})^{-1}  \vec{E}_1\vec{B}_0 \\
    \res_k(z) &:= \vec{V} - (\vec{H}-z\vec{I}) \vec{Q}_k (\vec{T}_k - z\vec{I})^{-1}  \vec{E}_1\vec{B}_0.
\end{align*}
\end{definition}

If $p$ is a polynomial with $\deg(p)<k$ then $p(\vec{H})\vec{V}$ is contained in $\mathcal{K}_k(\vec{H},\vec{V})$.
For Hermitian positive-definite $\vec{H}-z\vec{I}$, we then have the bound
\begin{align}
    \| \err_k(z) \|_{\vec{H}-z\vec{I}}
    &\leq \nonumber
    \min_{\deg(p)<k} \| (\vec{H}-z\vec{I})^{-1} \vec{V} - p(\vec{H})\vec{V} \|_{\vec{H}-z\vec{I}}
    \\&\leq \nonumber
    \min_{\deg(p)<k} \| (\vec{H}-z\vec{I})^{-1} - p(\vec{H}) \|_2 \| \vec{V} \|_{\vec{H}-z\vec{I}}
    \\&=  \| \vec{V} \|_{\vec{H}-z\vec{I}} \min_{\deg(p)<k} \max_{x\in\Lambda(\vec{H})} | (x-z)^{-1} - p(x) |.\label{eqn:minimax}
\end{align}
This bound depends strongly on the distribution of the eigenvalues of $\vec{H}$ and is typically far more representative of the true convergence of block-Lanczos-FA than \cref{eqn:unif_bound}.
Discussions on \cref{eqn:minimax} and why it is often more representative than \cref{eqn:unif_bound} are easy to find; see for instance \cite{carson_liesen_strakos_22} and the references within.

A posteriori error bounds and estimates for this setting are also simpler than the case of general $f$, even if $\vec{H}-z\vec{I}$ is not positive definite.
For instance, $\|\err_k(z)\|_{(\vec{H}-z\vec{I})^2} = \|\res_k(z)\|_\F$ is easily computed from the block-Lanczos-FA iterate.
The relationship between the error and residual block vector also allows us to derive the bound \begin{equation*}
    \|\err_k(w)\|_\F 
    \leq \| (\vec H-z\vec I)^{-1} \|_2 \|\res_k(z)\|_\F 
    = (\lambda_{\textup{min}}(\vec H) - z)^{-1} \|\res_k(z)\|_\F.
\end{equation*}
More precise bounds and estimates for other norms have been widely studied for the case $b=1$  \cite{strakos_tichy_02,meurant_tichy_18,estrin_orban_saunders_19,meurant_papez_tichy_21,meurant_tichy_22}.
Many of these techniques can be extended to the case $b>1$, and we provide an example of a simple error estimate for the $(\vec{H}-w\vec{I})$-norm in \cref{sec:cg_est}; see also \cite{schaerer_szyld_torrest_21}.

In order to derive a posteriori stopping criteria for Lanczos-FA on general $f$, it is common to use the fact that the output of the Lanczos-FA algorithm can be related to the error of Lanczos-FA used to approximate the solution $(\vec{H} - z\vec{I})^{-1}\vec{v}$ at varying values of $z$ \cite{frommer_simoncini_08b,frommer_simoncini_09,ilic_turner_simpson_09,frommer_kahl_lippert_rittich_13,frommer_guttel_schweitzer_14a,frommer_schweitzer_15,chen_greenbaum_musco_musco_22a,simunec_23}.
This allows a priori bounds such as \cref{eqn:minimax} or a posteriori bounds and estimates such as those discussed in the previous paragraph for Lanczos-FA on linear systems to be upgraded to bounds and estimates for general $f$.

\section{Analysis}
\label{sec:analysis}

Our main result is a generalization of \cite[Theorem 2.6]{chen_greenbaum_musco_musco_22a} for the block-Lanczos algorithm for matrix function approximation:
\begin{theorem}
\label{thm:main}
Fix a $n\times n$ Hermitian matrix $\vec{H}$ and $n\times b$ matrix $\vec{V}$.
Let $\vec{Q}_k$, $\vec{T}_k$, and $\vec{B}_0$ be outputs of \cref{alg:block_lanczos} as described in \cref{eqn:krylov_recurrence}.
Let $S\subset\mathbb{R}$ be any set with $\Lambda(\vec{H}) \subseteq S$, and for $w,z\in\mathbb{C}$, define $Q_S(w,z) := \sup_{x\in S} {|x-w|}/{|x-z|}$.
For $u\in\mathbb{C}$, define $\vec{C}(u) := -\vec E_k^\cT(\vec{T}_k-u\vec I)^{-1}\vec E_1\vec{B}_0$.
Finally, suppose $\Gamma$ is a union of simple closed curves that enclose $\Lambda(\vec H)$ and $\Lambda(\vec{T}_k)$, that $w\not\in \Lambda(\vec{T}_k)$, and that $f:\mathbb{C}\to\mathbb{C}$ is continuous on $\Gamma$ and analytic on the interior of each curve of $\Gamma$.
Then, with $\|\cdot\|$ indicating $\|\cdot\|_2$ or $\|\cdot\|_{h(\vec{H})}$ for some function $h:\R\to\R$ positive on $\Lambda(\vec{H})$, the block-Lanczos-FA iterate satisfies the error bound
\begin{equation*}
    \| f(\vec{H}) \vec{V} - \lan_k(f)\| 
    \leq \hspace{-.1em}
    \underbrace{\hspace{-.2em}\vphantom{ \bigg| }\left(  \frac{1}{2\pi} \oint_{\Gamma} |f(z)| \cdot Q_S(w,z) \cdot \| \vec{C}_k(w)^{-1} \vec{C}_k(z) \|_2 \cdot |\d{z}| \right)}_{\textup{integral term}}
    \cdot \hspace{-.4em} \underbrace{\vphantom{ \bigg| } \| \err_k(w) \|}_{\textup{linear system error}}\hspace{-.6em}.
\end{equation*}
\end{theorem}

\Cref{thm:main} allows us to bound the block-Lanczos FA error above by the product of an integral term and a linear system error term that can each be computed effectively numerically. 
Because the linear system term depends on spectral properties of $\vec{H}$ as well as the interaction between $\vec{H}$ and $\vec{V}$, \cref{thm:main} is able to capture the convergence of block-Lanczos-FA better than bounds such as \cref{eqn:unif_bound}.

In practice, the integral in \cref{thm:main} must be computed numerically.
Critically, however, the integrand \emph{does not depend on the dimension} $n$ of $\vec{H}$, and is therefore relatively cheap compared to to the costs of the block-Lanczos algorithm whenever $n$ is very large.

For each $z$, $\vec{C}_k(z)$ can be computed relatively cheaply from an eigen-decomposition of the $(kb)\times (kb)$ symmetric block-tridiagonal matrix $\vec{T}_k$.
Thus, the most costly part of evaluating the integrand is likely the norm-computation $\| \vec{C}_k(w)^{-1} \vec{C}_k(z) \|_2$ (although since this quantity is continuous in $z$, iterative methods like the power method and/or randomized methods like the randomized SVD \cite{halko_martinsson_tropp_11} may be effective for this task).

In the case $\vec{V}$ is a single vector ($b=1$), then \cref{thm:main} is  equivalent to a special case of \cite[Theorem 2.6]{chen_greenbaum_musco_musco_22a}.
In fact, in this case $\|\vec{C}_k(w)^{-1}\vec{C}_k(z)\|_2$ has a simple formula in terms of the eigenvalues of $\vec{T}_k$. Since the eigenvalues of $\vec{T}_k$ are known to interlace those of $\vec{H}$, this formula can be bounded a priori, resulting in a priori error bounds. 
For $b>1$ we were unable to derive a similarly simple formula or a priori bound for $\|\vec{C}_k(w)^{-1}\vec{C}_k(z)\|_2$ and hence have focused on a posteriori bounds.

Note that in the case $S$ is a single interval, $Q_S$ can be evaluated explicitly using \cite[Lemma 3.1]{chen_greenbaum_musco_musco_22a}:
\begin{lemma}
\label{thm:Q_wz_value}
    For any interval \( [a,b] \subset \mathbb{R} \), if \( z \in \mathbb{C} \setminus [a,b] \) and \( w\in\mathbb{R} \), we have
\begin{align*}
    Q_{[a,b]}(w,z)
    = \max \left\{ 
    \left| \frac{a-w}{a-z} \right|, 
    \left| \frac{b-w}{b-z} \right|,
    \left( \left| \frac{z-w}{\Im(z)} \right|~\text{ if } x^{\star} \in [a,b] ~\text{else}~0 \right)
    \right\},
\end{align*}
where 
\begin{align*}
    x^\star := \frac{\Re(z)^2 + \Im(z)^2 - \Re(z) w}{\Re(z)-w}.
\end{align*}
\end{lemma}

We now introduce some notation and two auxiliary results which will be useful in proving \cref{thm:main}.

\begin{lemma}
\label[lemma]{thm:res_prop_Qk1}
For all $z\in\mathbb{C}\setminus\Lambda(\vec{T}_k)$, $\res_k(z) =  \overline{\vec{Q}}_{k+1} \vec{B}_k \vec{C}_k(z)$, where $\vec{C}_k(z)$ is as defined in \cref{thm:main}.
\end{lemma}

We can then use this to relate the error and residual vectors systems $(\vec{H} - z\vec{I})\vec{X} = \vec{V}$ with those for $(\vec{H} - w\vec{I})\vec{X} = \vec{V}$.

\begin{definition}
    For any $w, z\in\mathbb{C}$, we define 
    \begin{align*}
        h_{z}(x) &:= \frac{1}{x-z}, \qquad
        h_{w, z}(x) &:= \frac{x-w}{x-z}.
    \end{align*} 
\end{definition}

\begin{corollary}
\label[corollary]{thm:err_res_shift}
For all $w\in\mathbb{C}$ where $\vec{C}_k(w)$ is invertible and for all $z\in\mathbb{C}\setminus\Lambda(\vec H)$,
\begin{align*}
    \err_k(z) &= h_{w,z}(\vec{H}) \err_k(w) \vec{C}_k(w)^{-1} \vec{C}_k(z)\\
    \res_k(z) &= \res_k(w) \vec{C}_k(w)^{-1} \vec{C}_k(z).
\end{align*}
\end{corollary}

Both \cref{thm:res_prop_Qk1,thm:err_res_shift} are well-known in the literature; see for instance \cite{frommer_lund_szyld_17,frommer_lund_szyld_20} (we provide the proofs in \cref{sec:appendix} for convenience).

\begin{proof}[Proof of \cref{thm:main}]
Since $\Gamma$ is a union of simple closed curves that encompasses $\Lambda(\vec H)$, and since $f$ is analytic on the interior of $\Gamma$ and continuous on the boundary, by the Cauchy integral formula, we find that 
\begin{equation}
f(\vec{H})\vec{V} = -\frac{1}{2\pi i}\oint_\Gamma f(z)(\vec{H}-z\vec{I})^{-1}\vec{V} \d z. \label{eqn:fAV_CIF}
\end{equation}
Since $\Gamma$ also encloses $\Lambda(\vec{T}_k)$, the block-Lanczos-FA approximation can be written as
\begin{equation}
\vec{Q}_kf(\vec{T}_k)\vec{E}_1\vec{B}_0 = -\frac{1}{2\pi i}\oint_\Gamma f(z)\vec{Q}_k(\vec{T}_k-z\vec I)^{-1} \vec{E}_1 \vec{B}_0\d{z}. \label{eqn:QfTE_CIF}
\end{equation}
Combining \cref{eqn:fAV_CIF,eqn:QfTE_CIF}, we find
\begin{align*}
    f(\vec{H}) \vec{V} - \vec{Q}_k f(\vec{T}_k) \vec{E}_1\vec{B}_0 &= -\frac{1}{2\pi i}\oint_{\Gamma}f(z)((\vec{H}-z\vec{I})^{-1}\vec{V} - \vec{Q}_k(\vec{T}_k-z\vec I)^{-1} \vec{E}_1 \vec{B}_0)\d z.
\end{align*}
Since $\lan_k(f) = \vec{Q}_k f(\vec{T}_k) \vec{E}_1 \vec{B}_0$, using basic properties of matrix norms and integrals, along with the inequality \cref{eqn:norm_ineq} for the case $\|\cdot\| = \|\cdot\|_{h(\vec{H})}$, 
\begin{align}
    \|f(\vec{H}) \vec{V} - \lan_k(f)\|
    \hspace{-4em}&\hspace{4em}= \left\|-\frac{1}{2\pi i}\oint_{\Gamma}f(z)\err_k(z)\d z\right\| \label{eqn:error_CIF}
    \\
    &\leq \frac{1}{2\pi} \oint_{\Gamma}\|f(z)\err_k(z)\|\cdot |\d z| \label{eqn:triangle_ineq}\\
    &= \frac{1}{2\pi} \oint_{\Gamma}|f(z)|\cdot \|h_{w,z}(\vec{H})\err_k(w)\vec{C}_k(w)^{-1} \vec{C}_k(z)\|\cdot |\d z| \nonumber\\
    &\leq 
    \label{eqn:final}\frac{1}{2\pi} \oint_{\Gamma}|f(z)|\cdot \|h_{w, z}(\vec{H})\|_2\cdot \|\vec{C}_k(w)^{-1} \vec{C}_k(z)\|_2\cdot |\d z|\cdot \|\err_k(w)\| . 
\end{align}
Finally, since $\Lambda(\vec{H})\subset S$, 
\begin{align*}
    \|h_{w, z}(\vec{H})\|_2
    &= \max_{x\in\Lambda(\vec{H})} \frac{|x-w|}{|x-z|}
    \leq \sup_{x\in S} \frac{|x-w|}{|x-z|}
    = Q_{S}(w,z).
\end{align*}
The result for the induced norm then follows by inserting this bound into \cref{eqn:final}.
\end{proof}

We remark that \cref{eqn:triangle_ineq} can be bounded or approximated directly using bounds/estimates for $\| \err_k(z) \|$.
If estimating $\|\err_k(z)\|$ involves computations that scale with the dimension $n$, approximating \cref{eqn:triangle_ineq} numerically could be significantly more expensive than the suggested bound in \cref{thm:main}.
In \cref{sec:experiments}, we show the quality of both bounds to illustrate the amount of slack introduced between them.

\subsection{Quadratic forms}
The diagonal entries of $\vec{V}^\cT f(\vec{H}) \vec{V}$ are quadratic forms involving the columns of $\vec{V}$.
These are widely used in a number of applications, with stochastic trace estimation being a particularly common example \cite{bai_fahey_golub_96,golub_meurant_09,schnack_richter_steinigeweg_20,meyer_musco_musco_woodruff_2021,chen_hallman_23,chen_cheng_22}.
It is common to approximate $\vec{V}^\cT f(\vec{H})\vec{V}$ with $\vec{V}^\cT \lan_k(f)$, and we can derive error bounds as
in \cref{thm:main} for an approximation similar to what is derived in \cite[\S6]{chen_greenbaum_musco_musco_22a}.
Since $\vec{H}$ is Hermitian, it is the case that $(\vec{H}-z\vec{I})^\cT = \vec{H}-\overline{z}\vec{I}$ and therefore that
\begin{equation*}
    \vec{V}^\cT(\vec{H}-z\vec{I})^{-1} = ((\vec{H}-\overline{z}\vec{I})^{-1}\vec{V})^\cT.
\end{equation*}
Then $h_{\overline{z}}(\vec{T}_k) = (\vec{T}_k-\overline{z}\vec{I})^{-1}$, so from \cref{def:block_LanczosFA} and \cref{def:err_and_res} we have
\begin{equation*}
    ((\vec{H}-\overline{z}\vec{I})^{-1}\vec{V})^\cT 
    = (\lan_k(h_{\overline{z}}) + \err_k(\overline{z}))^\cT.
\end{equation*}
Therefore, by \cref{thm:err_res_shift} the quadratic form error can be expanded to
\begin{align*}
    \vec{V}^\cT\err_k(z) &= \vec{V}^\cT(\vec{H}-z\vec{I})^{-1}\res_k(z)\\
    &= (\lan_k(h_{\overline{z}}) + \err_k(\overline{z}))^\cT\res_k(z)\\
    &= \lan_k(h_{\overline{z}})^\cT\res_k(z) + \err_k(\overline{z})^\cT\res_k(z).
\end{align*}
By \cref{def:block_LanczosFA}, we know that $\lan_k(h_{\overline{z}}) = \vec{Q}_k h_{\overline{z}}(\vec{T}_k) \vec{E}_1 \vec{B}_0$ so that $\lan_k(h_{\overline{z}})^\cT = \vec{B}_0^\cT\vec{E}_1^\cT h_{\overline{z}}(\vec{T}_k)^\cT\vec{Q}_k^\cT$. 
By \cref{thm:res_prop_Qk1}, we also know that $\res_k(z) =  \overline{\vec{Q}}_{k+1} \vec{B}_k \vec{C}_k(z)$. However, $\vec{Q}_k$ and $\overline{\vec{Q}}_{k+1}$ are orthogonal so that $\lan_k(h_{\overline{z}})^\cT\res_k(z) = \vec{0}$. 
Therefore, again using \cref{thm:res_prop_Qk1} we have that 
\begin{align*}
    \vec{V}^\cT\err_k(z) &= \err_k(\overline{z})^\cT\res_k(z)\\
    &= ((\vec{H}-\overline{z}\vec{I})^{-1}\res_k(\overline{z}))^\cT\res_k(z)\\
    &= ((\vec{H}-\overline{z}\vec{I})^{-1}\res_k(w)\vec{C}_k(w)^{-1}\vec{C}(\overline{z}))^\cT\res_k(w)\vec{C}_k(w)^{-1}\vec{C}_k(z).
\end{align*}
By definition, $\vec{C}(\overline{z}) = -\vec E_k^\cT(\vec{T}_k-\overline{z}\vec I)^{-1}\vec E_1\vec{B}_0$. 
Since $\vec{T}_k$ is real, we know that $(\vec{T}_k-\overline{z}\vec I)^{-1} = \overline{(\vec{T}_k-z\vec I)^{-1}}$ so that we have $\vec{C}(\overline{z}) = \overline{\vec{C}_k(z)}$.
This shows that $\|\vec{C}(\overline{z})\|_2 = \|\overline{\vec{C}_k(z)}\|_2$, so that
\[
\|\vec{V}^\cT\err_k(z)\|_2 
\leq \|(\vec{H}-\overline{z}\vec{I})^{-1}\|_2\cdot\|\res_k(w)\|_2^2\cdot\|\vec{C}_k(w)^{-1}\vec{C}_k(z)\|_2^2.
\] 
Define $\widetilde{Q}_S(z) := \sup_{x\in S} 1/|x-z|$.
Then, provided $\Lambda(\vec{H}) \subseteq S$, we have $\|(\vec{H}-\overline{z}\vec{I})^{-1}\|_2 = \|(\vec{H}-z\vec{I})^{-1}\|_2 \leq \widetilde{Q}_S(z)$.
Therefore, following the proof of \cref{thm:main}, we have
\begin{align}
    \nonumber\hspace{6em}\hspace{-6em}\|\vec{V}^\cT f(\vec{H}) \vec{V} - 
    \vec{V}^\cT \lan_k(f) \|_2 
    &= \left\|-\frac{1}{2\pi i}\oint_{\Gamma}f(z)\vec{V}^\cT\err_k(z)\d z\right\|_2\\
    &\label{eqn:QF_triangle}\leq \frac{1}{2\pi} \oint_{\Gamma}\|f(z)\vec{V}^\cT\err_k(z)\|_2\cdot |\d z| \\
    &\label{eqn:QF_main}\leq \bigg(\frac{1}{2\pi} \oint_{\Gamma}|f(z)|\cdot \widetilde{Q}_S(z) \cdot\|\vec{C}_k(w)^{-1}\vec{C}_k(z)\|_2^2\cdot |\d z| \bigg)\cdot\|\res_k(w)\|_2^2.
\end{align}
As with $Q_S(w,z)$, $\widetilde{Q}_S(z)$ can be easily computed when $S$ is a single interval \cite[Lemma 6.1]{chen_greenbaum_musco_musco_22a}:
\begin{lemma}
For any interval \( [a,b] \subset \mathbb{R} \), if \( z \in \mathbb{C} \setminus [a,b] \), we have
\[
\widetilde{Q}_S(z) = 
\begin{cases}
1/|\Im(z)| & \Re(z) \in [a,b], \\
1/|a-z| & \Re(z) < a, \\
1/|b-z| & \Re(z) > b.
\end{cases}
\]
\end{lemma}
The bound \cref{eqn:QF_main} is very similar to the bound in \cref{thm:main}, except $Q_S(w, z)$ is replaced with $\widetilde{Q}_S(z)$, $\|\vec{C}_k(w)^{-1}\vec{C}_k(z)\|_2$ is replaced by $\|\vec{C}_k(w)^{-1}\vec{C}_k(z)\|_2^2$, and $\|\err_k(w)\|$ is replaced by $\|\res_k(w)\|_2^2$. 
Thus, we anticipate that the quadratic form will converge at a rate that is about twice as fast as the norm of the matrix function's error, mirroring the block size one case.
Note that we have also used the operator norm, which gives an entry-wise bound for the block-Lanczos-FA error.

\subsection{Finite precision arithmetic}
\label{sec:analysis_fp}

In exact arithmetic, \cref{alg:reorth} in \cref{alg:block_lanczos} is unnecessary as $\vec{Z}$ is already orthogonal to $\overline{\vec{Q}}_0, \ldots, \overline{\vec{Q}}_{j-1}$. 
However, in finite precision arithmetic, omitting reorthogonalization results in behavior very different than if reorthogonalization is used.
Even without reorthogonalization, algorithms such as Lanczos-FA are often still effective in practice. 
This has been established rigorously for the block size one case \cite{druskin_knizhnerman_91,musco_musco_sidford_18}.
In \cite[Section 5.1]{chen_greenbaum_musco_musco_22a} a rounding error analysis for error bounds similar to the one in this paper is provided for the block size one case.
However, such an analysis is predicated on Paige's analysis of the Lanczos algorithm in finite precision arithmetic \cite{paige_71,paige_76,paige_80}.
While we are not aware of any analyses similar to that of Paige for the block-Lanczos algorithm, it seems reasonable to assume that the outputs of the block-Lanczos algorithm will often satisfy a perturbed version of the block-three-term recurrence \cref{eqn:krylov_recurrence}.

Thus, we will perform an analysis of our bounds under the assumption that the outputs of the Lanczos algorithm run in finite precision arithmetic satisfy a perturbed three-term recurrence
\begin{align}
\label{eqn:krylov_recurrence_fp}
    \vec{H} \vec{Q}_k = \vec{Q}_k \vec{T}_k + \overline{\vec{Q}}_{k+1} \vec{B}_k \vec{E}_k^\cT + \vec{F}_k,
\end{align}
where the perturbation term $\vec{F}_k$ is \emph{assumed} to be small.
We remark that even without a priori bounds on the size of $\vec{F}_k$, the size of $\vec{F}_k$ can be computed after the algorithm has been run, and in all cases in which we have observed, $\vec{F}_k$ is indeed small. 

We now reproduce analogs of \cref{thm:res_prop_Qk1}, \cref{thm:err_res_shift}, \cref{thm:main} for the Lanczos algorithm in finite precision arithmetic.
The perturbed Lanczos factorization \cref{eqn:krylov_recurrence_fp} can be shifted into
\[(\vec H-z\vec I)\vec Q_k = \vec Q_k(\vec{T}_k-z\vec I) + \overline{\vec{Q}}_{k+1}\vec B_k\vec E_k^\cT + \vec F_k.\] 
Thus, substituting into \cref{def:err_and_res}, we have
\begin{align*}
    \res_k(z) &:= \vec{V} - (\vec{H}-z\vec{I}) \vec{Q}_k (\vec{T}_k - z\vec{I})^{-1}  \vec{E}_1\vec{B}_0\\
    &= \vec{V} - (\vec Q_k(\vec{T}_k-z\vec I) + \overline{\vec{Q}}_{k+1}\vec B_k\vec E_k^\cT + \vec F_k) (\vec{T}_k - z\vec{I})^{-1}  \vec{E}_1\vec{B}_0\\
    &= \vec{V} - \vec Q_k  \vec{E}_1\vec{B}_0 - \overline{\vec{Q}}_{k+1}\vec B_k\vec E_k^\cT(\vec{T}_k - z\vec{I})^{-1}  \vec{E}_1\vec{B}_0 - \vec F_k(\vec{T}_k - z\vec{I})^{-1}  \vec{E}_1\vec{B}_0.
\end{align*}
Since $\vec{E}_1 = \left[ \vec I_{b}, \vec 0, ..., \vec 0 \right]^\cT$, where $\vec I_{b}$ is the $b\times b$ identity matrix, we know that $\vec Q_k \vec E_1\vec{B}_0 = \vec Q_1\vec{B}_0 = \vec V$.  By substitution, we have
\begin{align*}
    \res_k(z) &= \vec{V} - \vec Q_k  \vec{E}_1\vec{B}_0 - \overline{\vec{Q}}_{k+1}\vec B_k\vec E_k^\cT(\vec{T}_k - z\vec{I})^{-1}  \vec{E}_1\vec{B}_0 - \vec F_k(\vec{T}_k - z\vec{I})^{-1}  \vec{E}_1\vec{B}_0\\
    &= -\overline{\vec{Q}}_{k+1}\vec B_k\vec E_k^\cT(\vec{T}_k - z\vec{I})^{-1}  \vec{E}_1\vec{B}_0 - \vec F_k(\vec{T}_k - z\vec{I})^{-1}  \vec{E}_1\vec{B}_0\\
    &= \overline{\vec{Q}}_{k+1}\vec B_k\vec{C}_k(z) - \vec F_k(\vec{T}_k - z\vec{I})^{-1}  \vec{E}_1\vec{B}_0.
\end{align*}
This closely matches \cref{thm:res_prop_Qk1}, with the additional term $ - \vec F_k(\vec{T}_k - z\vec{I})^{-1}  \vec{E}_1\vec{B}_0$ that scales with the size of the perturbation term $\vec F_k$. Furthermore, we have
\begin{align*}
    \res_k(z) &= \overline{\vec{Q}}_{k+1}\vec B_k\vec{C}_k(z) - \vec F_k(\vec{T}_k - z\vec{I})^{-1}  \vec{E}_1\vec{B}_0\\
    &= \overline{\vec{Q}}_{k+1}\vec B_k\vec{C}_k(w)\vec{C}_k(w)^{-1}\vec{C}_k(z) - \vec F_k(\vec{T}_k - z\vec{I})^{-1}  \vec{E}_1\vec{B}_0\\
    &= \res_k(w)\vec{C}_k(w)^{-1}\vec{C}_k(z) + (\vec F_k(\vec{T}_k - w\vec{I})^{-1}  \vec{E}_1\vec{B}_0\vec{C}_k(w)^{-1}\vec{C}_k(z) 
    \\& \hspace{20em} - \vec F_k(\vec{T}_k - z\vec{I})^{-1}  \vec{E}_1\vec{B}_0).
\end{align*}
This can be written into
\begin{equation}
\label{eqn:res_fp}
    \res_k(z) = \res_k(w)\vec{C}_k(w)^{-1}\vec{C}_k(z) + \vec f_k(w, z),
\end{equation}
where we have defined
\begin{equation}
\label{eqn:fk}
    \vec f_k(w, z) = \vec F_k(\vec{T}_k - w\vec{I})^{-1}  \vec{E}_1\vec{B}_0\vec{C}_k(w)^{-1}\vec{C}_k(z) - \vec F_k(\vec{T}_k - z\vec{I})^{-1}  \vec{E}_1\vec{B}_0.
\end{equation}
Similarly, since $\err_k(z) = (\vec{H}-z\vec{I})^{-1}\res_k(z)$, we have
\begin{align}
    \err_k(z) 
    &= (\vec{H}-z\vec{I})^{-1}\res_k(w)\vec{C}_k(w)^{-1}\vec{C}_k(z) + (\vec{H}-z\vec{I})^{-1}\vec f_k(w, z) \nonumber \\
    &= h_{w, z}(\vec{H})\err_k(w)\vec{C}_k(w)^{-1} \vec{C}_k(z) + (\vec{H}-z\vec{I})^{-1}\vec f_k(w, z).
    \label{eqn:err_fp}
\end{align}
Expressions \cref{eqn:res_fp,eqn:err_fp} are similar to those in \cref{thm:err_res_shift}, with the difference that for residual there is an additional additive term $\vec f_k(w, z)$ while for error there is an additional additive term $(\vec{H}-z\vec{I})^{-1}\vec f_k(w, z)$. 

Using this, we now derive an analogue of \cref{thm:main} under finite arithmetic. 
From \eqref{eqn:error_CIF}, with $\|\cdot\|$ indicating $\|\cdot\|_2$ or $\|\cdot\|_{h(\vec{H})}$ for some function $h:\R\to\R$ positive on $\Lambda(\vec{H})$, we have
\begin{align}
    \|f(\vec{H}) \vec{V} - \lan_k(f) \|
    &= \left\|-\frac{1}{2\pi i}\oint_{\Gamma}f(z)\err_k(z)\d z\right\| \nonumber
    \\&= \frac{1}{2\pi}\left\|\oint_{\Gamma}f(z)h_{w,z}(\vec{H})\err_k(w)\vec{C}_k(w)^{-1} \vec{C}_k(z)\d z\right. 
    \nonumber\\&\hspace{10em} \left.+ \oint_{\Gamma}f(z)(\vec{H}-z\vec{I})^{-1}\vec f_k(w, z)\d z\right\| \nonumber
    \\&\leq \frac{1}{2\pi}\left\|\oint_{\Gamma}f(z)h_{w,z}(\vec{H})\err_k(w)\vec{C}_k(w)^{-1} \vec{C}_k(z)\d z\right\| 
    \label{eqn:main_fp}
    \\\nonumber&\hspace{10em}+ \frac{1}{2\pi}\left\|\oint_{\Gamma}f(z)(\vec{H}-z\vec{I})^{-1}\vec f_k(w, z)\d z\right\|.
\end{align}
The first term in \cref{eqn:main_fp} is identical to a term bounded in the proof of \cref{thm:main}.
Thus, in finite precision arithmetic, we obtain a bound similar to \cref{thm:main} with an additional integral term given by \[\frac{1}{2\pi}\left\|\oint_{\Gamma}f(z)(\vec{H}-z\vec{I})^{-1}\vec f_k(w, z)\d z\right\|.\]
Note that $\vec{f}_k(w,z)$, defined in \cref{eqn:fk}, depends linearly on $\vec{F}_k$.
Thus, if $\vec{F}_k$ is small then this term will be small; i.e. \cref{thm:main} holds to a close degree under finite precision arithmetic.

\section{Numerical Experiments}\label{sec:experiments}
Throughout this section, in our numerical experiments, we will compute $\|\err_k(z)\|$ directly using the spectral decomposition of $\vec H$, thus allowing us to shift our focus onto the evaluation of the integral term and the quantitative behavior of our bound. In practice, $\|\err_k(z)\|$ can be bounded or estimated using the techniques mentioned in \cref{sec:past_work}. Furthermore, we will use $h(\vec H) = \vec{H}-w\vec{I}$ for a given $w$ as the function for the induced norm, namely $\|\cdot\| = \|\cdot\|_{\vec{H}-w\vec{I}}$, unless specified otherwise.

We compute integral with SciPy's \texttt{integrate.quad} integrator, which is a wrapper for QUADPACK routines. 
In most cases, we let the integrator choose the points to evaluate the integral automatically. 
In some situations, the integrand may have most of its mass near a single point, and we pass a series of breakpoints to the integrator to obtain a more accurate value.
In our experiments we take $S = [\lambda_{\textup{min}},\lambda_{\textup{max}}]$ in which case $Q_S(w,z) := \sup_{x\in S}|x-w|/|x-z|$ can be computed efficiently using \cref{thm:Q_wz_value}. 
While $\lambda_{\textup{min}}$ and $\lambda_{\textup{max}}$ are typically unknown, they can be estimated from the eigenvalues of $\vec{T}_k$.
We focus on the induced norm $\|\cdot\|_{h(\vec{H})}$ (which is similar to a weighted Frobenius norm) rather than the operator norm $\|\cdot\|_2$, as most of the use-cases of block-Lanczos-FA which we are familiar with tend to have the error distributed uniformly across the columns of the output.
Regardless, due to the equivalence of the operator and Frobenius norms, they cannot differ significantly if the block size $b$ is not too large and $h(\vec{H})$ is not too poorly conditioned.

\subsection{Impact of block size} \label{sec:BS}

\begin{figure}
    \centering
    \includegraphics[scale = .6]{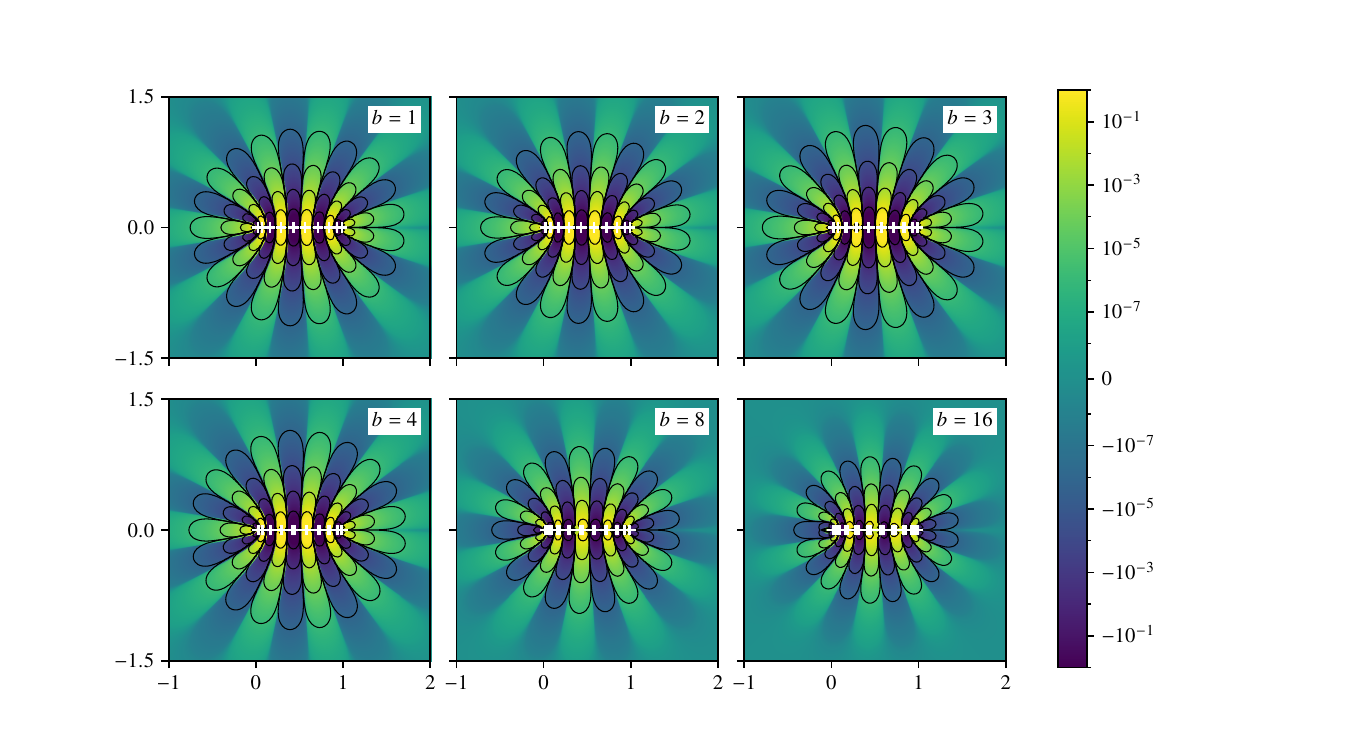}
    \caption{Imaginary part of the top-left entry of $f(z)\err_k(z)$ near the spectrum of $\vec H$ for block size $b = 1, 2, 3, 4, 8, 16$.
    Note that larger oscillations will result in more slack in the triangle inequality used between \cref{eqn:error_CIF} and \cref{eqn:triangle_ineq}.}
    \label{fig:im}
\end{figure}

There are two main sources of slack in our bound.
The first is from \cref{eqn:error_CIF} to \cref{eqn:triangle_ineq} due to losses associated with bounding the norm of an integral with the integral of a norm, and the second is from \cref{eqn:triangle_ineq} to the result of \cref{thm:main} due to potential slackness in a sub-multiplicative bound.
With block size introduced to the Lanczos algorithm, an interesting question that arises is the impact of block size on the bound given by \cref{thm:main}. 
In our first experiment, we illustrated relevant quantities for a range of block sizes to illustrate how the block size impacts our bound. For experiments in this section, we use the square root function $f(x) = \sqrt x$. We take $\vec{V}$ to be an $n\times b$ block-vector with independent standard Gaussian entries, and $w = 0$. We set $\vec{H}$ to be a $1000\times 1000$ diagonal matrix with linearly spaced diagonal elements between $10^{-2}$ and 1. We will investigate the slackness in our bound across block sizes $b = 1, 2, 3, 4, 8, 16$. Reorthogonalization is used.

We first study the loss from the triangle inequality between \cref{eqn:error_CIF} and \cref{eqn:triangle_ineq}. 
The error introduced by triangle inequality is larger when the function is more oscillatory over the contour. 
In fact, by the Cauchy--Schwarz inequality, \cref{eqn:error_CIF} and \cref{eqn:triangle_ineq} would be identical only if both real and imaginary parts of the integrand of \cref{eqn:error_CIF} were constant at every point on the contour.

In \cref{fig:im}, we plot the imaginary part of the top-left entry of $f(z)\err_k(z)$ on $-1 \leq \Re(z) \leq 2$ and $0\leq \Im(z) \leq 1.5$, and contour levels of $f(z)\err_k(z)$ at various levels.
From \cref{fig:im}, we observe that the sign of this entry of the integrand of \cref{eqn:error_CIF} oscillates as rays shoot out of eigenvalues of $\vec H$. Furthermore, the oscillation has a larger magnitude near the $\Lambda(\vec H)$, and the magnitude decreases as we move away from $\Lambda(\vec H)$. 
Most importantly, the qualitative behavior of the plot is similar across block sizes, so block size is not a factor that significantly impacts how we should choose the contour $\Gamma$ to reduce the error introduced by the triangle inequality. 
The real and imaginary parts of other entries of $f(z)\err_k(z)$ exhibit similar behavior.

\begin{figure}
    \centering
    \includegraphics[scale=.6]{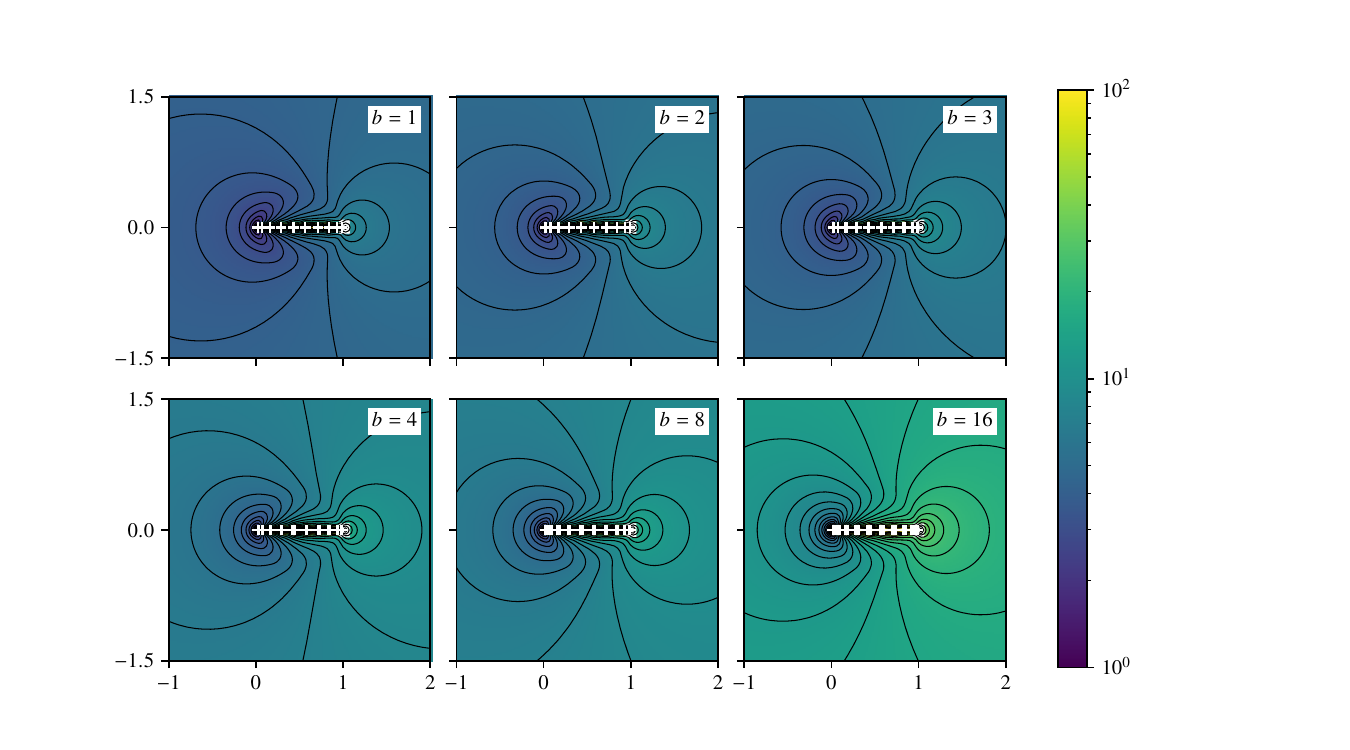}
    \caption{Ratio $T(z)$ defined in \cref{eqn:slack ratio} (with $w = \lambda_{\textup{min}}/100$) near the spectrum of $\vec H$ for block sizes $b = 1, 2, 3, 4, 8, 16$.
    Note that larger values of $T(z)$ result in more slack between \cref{eqn:triangle_ineq} and the result of \cref{thm:main}.}
    \label{fig:slack}
\end{figure}

To evaluate the difference between the bound defined in \cref{eqn:triangle_ineq} and \eqref{eqn:final}, we define 
\begin{equation}
\label{eqn:slack ratio}
T(z) = {\|h_{w,z}(\vec{H})\|_2 \| \vec{C}_k(w)^{-1} \vec{C}_k(z) \|_2 \| \err_k(w) \|_{\vec{H}-w\vec{I}}}{\|\err_k(z) \|_{\vec{H}-w\vec{I}}^{-1}}.
\end{equation}
The value of $T(z)$ at a point $z$ is the ratio between the integrands of \cref{eqn:triangle_ineq} and \eqref{eqn:final}.
If $T(z) = 1$ for all points on the contour, then \cref{eqn:triangle_ineq} and \cref{eqn:final} would be identical. Since \cref{eqn:final} is identical to the result of \cref{thm:main}, this ratio at a value $z$ offers us insights into the local difference between our bound and \cref{eqn:triangle_ineq}.

We plot the ratio $T(z)$ under the same condition as \cref{fig:im}. 
Quantitatively, as block size $b$ increases, the ratio increases. This indicates that large block sizes render our error bound less precise under the same contour and choice of $w$, but the optimal contour $\Gamma$ is similar across various block sizes $b$. Furthermore, the slack ratio has a singularity and is large near the eigenvalue of $\vec{H}$, while a small ratio is observed near $w = 0$. As we move away from eigenvalues of $\vec H$ or toward $w = 0$, the ratio decreases. 

\subsection{Impact of contour}

\begin{figure}
    \centering
    
    \begin{tikzpicture}
        \draw[thick] (0,0) -- ([shift=(-160:2cm)]0,0);
        \draw[thick] (0,0) -- ([shift=(160:2cm)]0,0);
        \begin{scope}[shift={(70:-.2cm)}]
            \draw[<->] (0,0) -- node [midway, below left] {$R$} ([shift=(160:2cm)]0,0);
        \end{scope}
        \draw[dotted] (0,0) -- (70:-.5cm);
        \draw[dotted] ([shift=(160:2cm)]0,0) -- ([shift=(160:2cm)]70:-.5cm);
        \draw[thick] ([shift=(-160:2cm)]0,0) arc (-160:160:2cm);
        \draw[<->] ([shift=(0:.5cm)]0,0) arc (0:160:.5cm) node [midway,above right] {$\Theta$};
        \fill (0,0) circle (2pt) node [below right] {$O$};
        \draw[dotted] (0,0) -- (1,0);
        \end{tikzpicture}
    \caption{Pac-Man contour with parameters $R$, $\Theta$, $O$}
    \label{fig:pacman}
\end{figure}
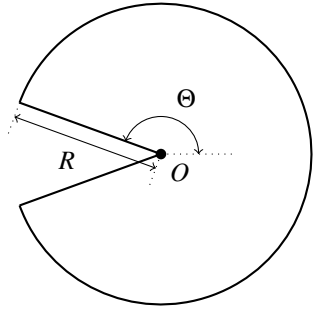

\begin{figure}
    \centering
    \includegraphics[scale=.6]{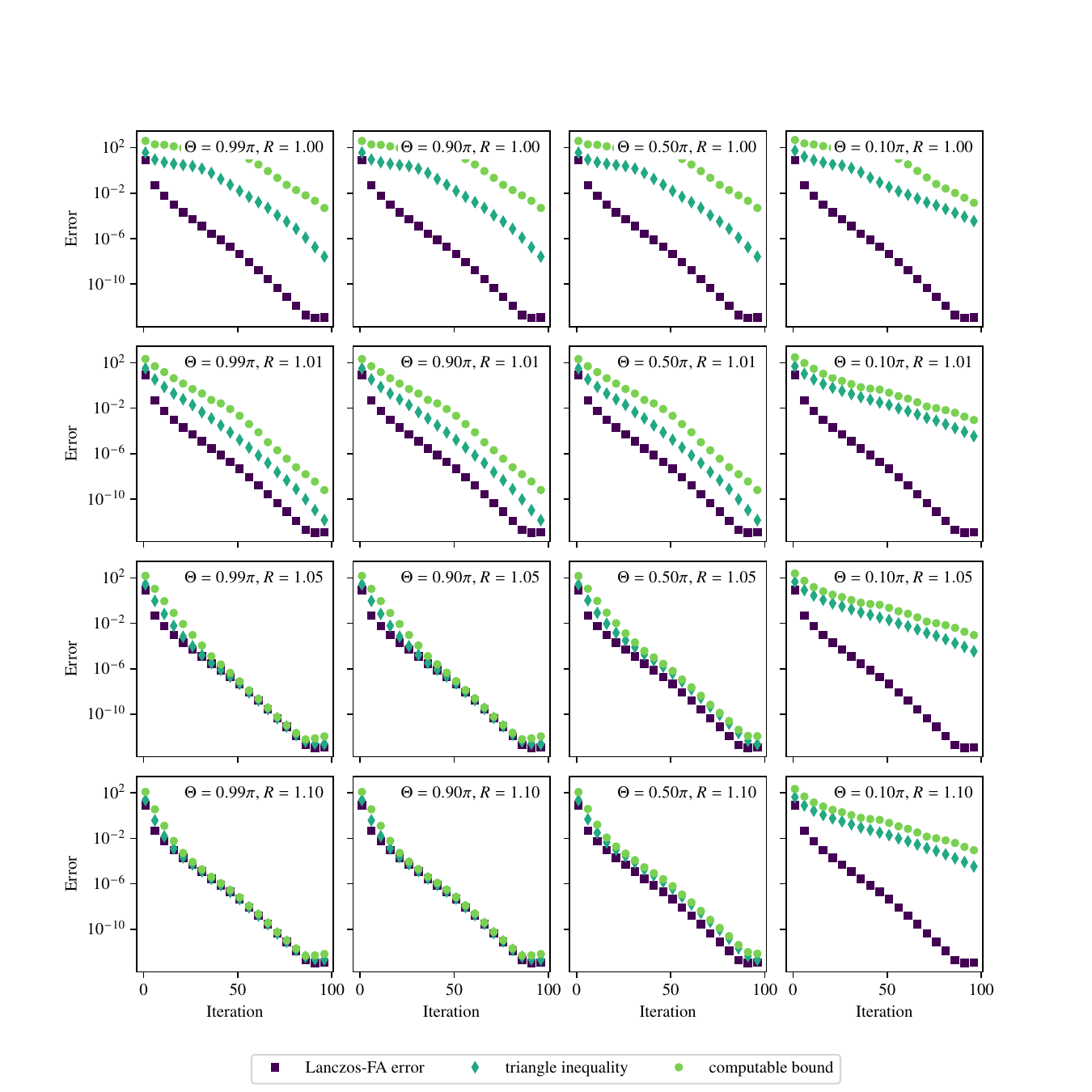}
    \caption{Performance of the Pac-Man contour across various parameters $R$ and $\Theta$ for $f(x) = \sqrt{x}$.
    The block-Lanczos-FA error is measured with $w = 0$.
    The bound ``triangle inequality'' is \cref{eqn:triangle_ineq} and the bound ``computable bound'' is \cref{thm:main}, both evaluated numerically.
    When $R$ and $\Theta$ are both reasonably large, the quality of the bounds is extremely good, but for some parameter choices, the bounds may have significant slack.}
    \label{fig:RvT}
\end{figure}

We now explore a family of contours and observe their impact on our computable bound in \cref{thm:main}. Specifically, we will be focusing on the ``Pac-Man contour'', which is the section of the circle $Re^{i\theta}+O$ with $-\Theta\leq \theta\leq \Theta$ and the lines between $O$ and $Re^{\pm i{\Theta}}$, for some $0< \Theta< \pi$. 
This contour is shown in \cref{fig:pacman}.
To visualize the slackness of this family of contours, we plot the block-Lanczos-FA error \cref{eqn:fAV_CIF}, triangle inequality error \cref{eqn:triangle_ineq}, and our computable bound in \cref{thm:main} across block-Lanczos algorithm iterations between 0 and 100.

We can try to minimize the losses described in the previous section by optimizing the parameters $R$ and $\Theta$ for the Pac-Man contour. 
From \cref{fig:im}, since the magnitude of the oscillation decreases as the point of evaluation moves further from the spectrum of $\vec H$, it seems reasonable that a larger radius $R$ leads to a smaller triangle inequality error.
Similar behavior is observed from \cref{fig:slack}. 
Owing to the symmetry across the real axis, it seems reasonable to choose a real origin. Since the square root function has a branch cut at the negative real axis and the origin is a point for the integral term to be evaluated, the origin cannot be on the negative real axis. But making it too close to the spectrum will also cause our bounds to deteriorate. 

According to the constraints above, we choose our origin to be $O = {\lambda_{\textup{min}}}/{100}$, and allow $R$ and $\Theta$ to vary with $R > \lambda_{\textup{max}}-O$, and $\pi < \theta < 0$.
We then evaluate the performance of \cref{thm:main} with the same matrix $\vec H$, block-vector $\vec V$, and parameter $w = 0$ as in \cref{sec:BS}, and show the results in \cref{fig:RvT}.
In the previous experiments, we have established empirically that block size does not have a large qualitative impact on the behavior of \cref{thm:main} for this example. 
Thus, we choose a fixed block size $b = 4$ to focus on the impacts of the contour.

\begin{figure}
    \centering
    \includegraphics[scale=.6]{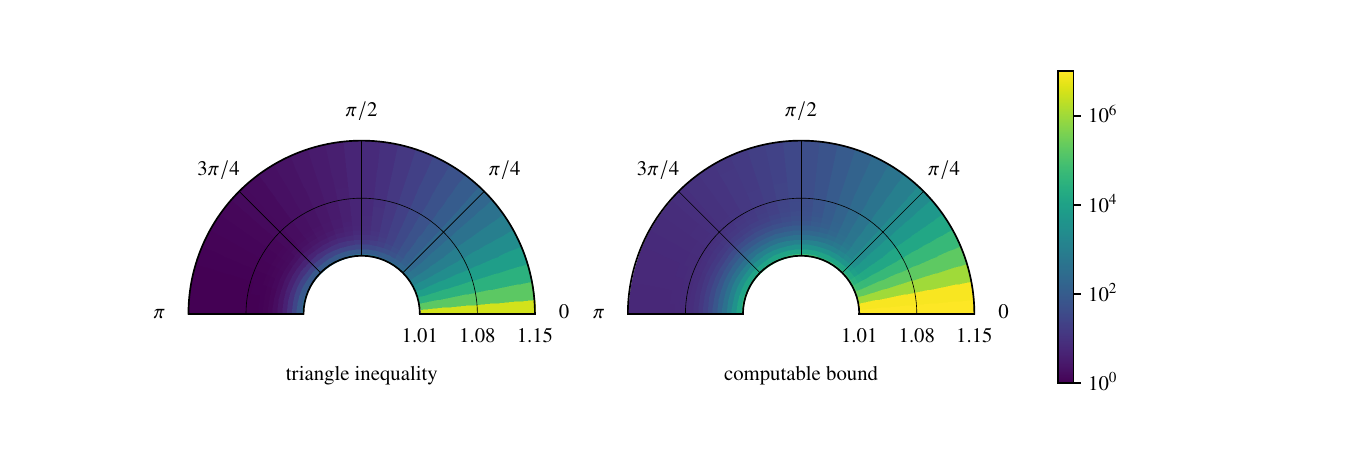}
    \caption{Bound divided by block-Lanczos-FA for the Pac-Man contour across various parameters $R$ and $\Theta$ for $f(x) = \sqrt{x}$ at $k=30$.
    The block-Lanczos-FA error is measured with $w = 0$.
    The bound ``triangle inequality'' is \cref{eqn:triangle_ineq} and the bound ``computable bound'' is \cref{thm:main} evaluated numerically.}
    \label{fig:phase}
\end{figure}

Overall, as $k$ increases, the computable bound remains near or converges to the triangle inequality error, while the triangle inequality error remains near or diverges from the block-Lanczos-FA error. Across each column of \cref{fig:RvT}, we observe that with a larger radius, both the triangle inequality error and computable bound converge to block-Lanczos-FA error, offering a tighter bound, which matches our observations from \cref{fig:im} and \cref{fig:slack}. Across each row of \cref{fig:RvT}, with larger $\Theta$ value, both triangle inequality error and computable bound move further from block-Lanczos-FA error, less to less accurate bound. This is because by the construction of the Pac-Man contour, with smaller $\Theta$ values (especially with $0 < \Theta < \frac{\pi}{2}$), the line between $O$ and $Re^{i{\Theta}}$ and between $O$ and $Re^{-i{\Theta}}$ move closer to the spectrum of $\vec H$. This increases the triangle inequality error and computable bound as observed in \cref{fig:im} and \cref{fig:slack}, which gives us the behavior observed in the last column of \cref{fig:RvT}.

To better visualize our claim, we focus on the performance of \cref{thm:main} with the block-Lanczos algorithm at $k = 30$ iterations. Under the same construction as \cref{fig:RvT}, we plot the ratio between triangle inequality error and block-Lanczos-FA error, and between computable bound and block-Lanczos-FA error across different values of $R$ and $\Theta$ as \cref{fig:phase}. 
We observe that a larger radius leads to a ratio closer to 1 between triangle inequality error and block-Lanczos-FA error and between computable bound and block-Lanczos-FA error, which matches our observation in \cref{fig:RvT}. Furthermore, in general, smaller values of $\Theta$ lead to worse performance. However, when $\Theta > \frac{\pi}{2}$, the performance is rather similar across different values of $\Theta$. 
This matches our expectations and suggests that our bounds do not require significant tuning.

These experiments suggest that, for the given problem, setting $R$ large is advantageous. 
In fact, nothing in the theory prohibits sending $R\to\infty$. 
Numerically, however, this may require more care, and the bounds are already pretty good for moderately large $R$.

\subsection{Sign Function}

We now study the behavior of our bound on the sign function, another function widely used in a number of application areas.
In particular, we consider the approximation of $\operatorname{sign}(\vec{H})\vec{V}$, where $\vec{H}$ is the Wilson fermion matrix from Quantum Chromodynamics (QCD). 
This quantity is required as part of a larger algorithm in QCD \cite{eshof_frommer_lippert_schilling_van_der_vorst_02,frommer_simoncini_08b}.

For this example, we use $\vec{H} = \vec{P}(\vec{I}-\frac{4}{3}k\vec{D})$, where $\vec{D}$ is from the QCD collection of the matrix market with file name \texttt{conf5.0-00l4x4-1000.mtx} with $k = 0.20611$. $\vec{P}$ is the permutation matrix given by 
\[\vec{P} = \vec{I}_3\otimes\begin{pmatrix}0 & 0 & 1 & 0 \\ 0 & 0 & 0 & 1 \\1 & 0 & 0 & 0 \\0 & 1 & 0 & 0 \end{pmatrix}\otimes \vec{I}_{256}.\]

Let $f(x)$ be the step function; $f(x) = 1$ for $x\geq 0$ and $f(x) = 0$ for $x<0$. Then $\operatorname{sign}(x) = 2f(x)-1$.
Since block-Lanczos-FA is linear in the input function, and since the algorithm applies low-degree polynomials (such as $x\mapsto -1$) exactly, 
\[
\operatorname{sign}(\vec{H})\vec{V} - 
\lan_k(\operatorname{sign})
= 2 f(\vec{H}) \vec{V} - \vec{V} - \big( 2\lan_k(f) - \vec{V} \big)
= 2 \big(f\vec{A})\vec{V} - \lan_k(f)\big).
\]
Thus, it suffices to study the behavior of block-Lanczos-FA on the step function.

In order to apply \cref{thm:main}, we require $f$ to be analytic on the interiors of each of the closed curves in the contour $\Gamma$.
Extend $f$ to the complex plane by $f(z) = 1$ for all $z$ such that $\Re(z) \geq 0$, $f(z) = 0$ for all $z$ such that $\Re(z) < 0$.
Assuming $\Lambda(\vec{H})$ and $\Lambda(\vec{T}_k)$ are bounded away from zero, we can take $\Gamma$ as the union of two simple closed curves; one encircling $[\lambda_{\textup{min}},-\epsilon]$ and one encircling $[\epsilon,\lambda_{\textup{max}}]$, where $\epsilon>0$ is some sufficiently small value.
When we apply \cref{thm:main}, the first contour can be ignored, as $f(z) = 0$ for all $z$ on the contour.

\begin{figure}
    \centering
    \includegraphics[scale=.6]{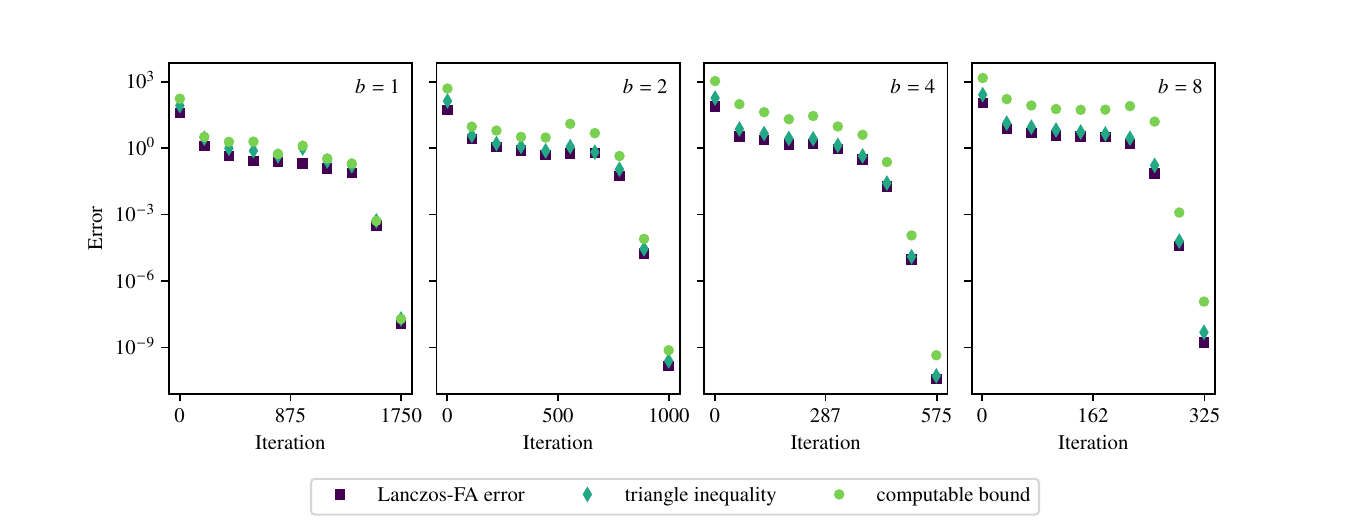}
    \caption{
    Convergence of the block-Lanczos algorithm for block sizes $b = 1, 2, 4, 8$ for the step function $f(x)=1$ if $x>0$ and $f(x) = 0$ if $x\leq 0$. 
    The contour used is the Pan-Man contour centered at 0 with $R = 2\lambda_{\textup{max}}$ and $\Theta = \frac{1}{2}\pi$. 
    The block-Lanczos-FA error is measured using $\|\cdot\|_{\vec{I}} = \| \cdot \|_\F$.
    The bound ``triangle inequality'' is \cref{eqn:triangle_ineq} and the bound ``computable bound'' is \cref{thm:main}, both evaluated numerically.
    While there is some deterioration of the bounds as the block size increases, the computable bound appears suitable for use as a stopping criterion for all of the block sizes we tested.
    }
    \label{fig:QCD_bsize}
\end{figure}

We investigate the impact of block size on this example problem using a Pac-Man contour centered at $0$ with $\Theta = 0.5\pi$ and a radius $R=2\lambda_{\textup{max}}$ large enough to cover all eigenvalues of $\vec{H}$. Note that with $\Theta = 0.5\pi$, the Pac-Man contour shapes like the letter "D" and never intersects with the imaginary axis. 
In \cref{fig:QCD_bsize}, we plot the rate of convergence of the block-Lanczos algorithm and our bound \cref{thm:main}, where $\vec{V}$ is a block-vector of random Gaussian entries. 
We note that the number of iterations required to converge is almost inversely proportional to the block size. 
This shows that having larger block sizes significantly decreases the number of matrix loads without increasing the total number of matrix-vector products significantly. 
Thus, bounds such as \cref{eqn:unif_bound}, which would not change significantly in the different cases, cannot be descriptive of the convergence.
On this example, our bounds are qualitatively similar in all cases, degrading slightly with the block size for the range of block sizes tested.
In this particular example, such degradation is not of major concern, as convergence is very sudden after a certain critical number of iterations have been performed. 

\subsection{Quadratic forms}

In this section we use the same matrix as in \cref{sec:BS} with $f(x) = 1/\sqrt{x}$ and compute bounds for the quadratic form error $\| \vec{V}^\cT f(\vec{H}) \vec{V} - \vec{V}^\cT \lan_k(f) \|_2$.
We use a Pac-Man contour centered at $0$ with $\Theta = 0.5\pi$ and a radius $R=2\lambda_{\textup{max}}$ and test our bounds for several block sizes.
The results of our experiments are shown in \cref{fig:QF_invsqrt_bsize}.
As in other experiments, we observe qualitatively similar results in all cases, with some deterioration as the block size increases. 

\begin{figure}    \centering
    \includegraphics[scale=.6]{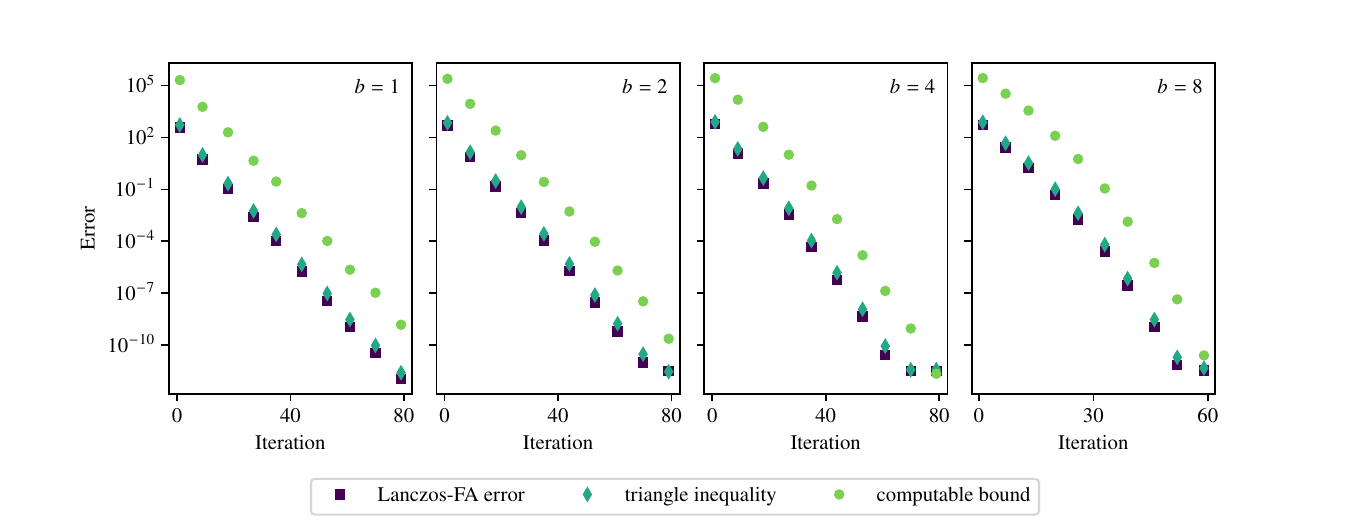}
    \caption{Convergence of the block-Lanczos algorithm for ``quadratic forms'' for block sizes $b = 1, 2, 4, 8$ for $f(x) = 1/\sqrt{x}$.
    The contour used is the Pan-Man contour centered at 0 with $R = 2\lambda_{\textup{max}}$ and $\Theta = \frac{\pi}{2}$. 
    The block-Lanczos-FA error is measured using $\|\cdot\|_{2}$. The bound ``triangle inequality'' is \cref{eqn:QF_triangle} and the bound ``computable bound'' is \cref{eqn:QF_main}, both evaluated numerically.}
    \label{fig:QF_invsqrt_bsize}
\end{figure}

\subsection{Finite precision arithmetic}

In this experiment, we consider the impact of reorthogonalization in the block-Lanczos algorithm on the quality of our bounds.
We again take $\vec{V}$ to have $b=4$ columns with independent Gaussian entries and set the eigenvalues of \( \vec{H} \) to those of the model problem \cite{strakos_91,strakos_greenbaum_92}.
Explicitly, the eigenvalues are
\begin{equation*}
    \lambda_1 = 1/\kappa
    ,\qquad \lambda_n = 1
    ,\qquad \lambda_i = \lambda_1 + \left( \frac{i-1}{n-1} \right) \cdot (\lambda_n -\lambda_1) \cdot \rho^{n-i}
    ,\qquad i=2,\ldots,n-1.
\end{equation*}
In our experiment, we use parameters \( n = 500 \), \( \kappa = 10^3 \), and \( \rho = 0.9 \) and set $f(x) = 1/\sqrt{x}$.
We run the block-Lanczos-FA algorithm to convergence with and without reorthogonalization and report the results of our bounds with a Pac-Man contour with $\Theta = \pi/100$ and $R = 2$ in \cref{fig:RvT_fp}.

\begin{figure}
    \centering
    \includegraphics[scale=.6]{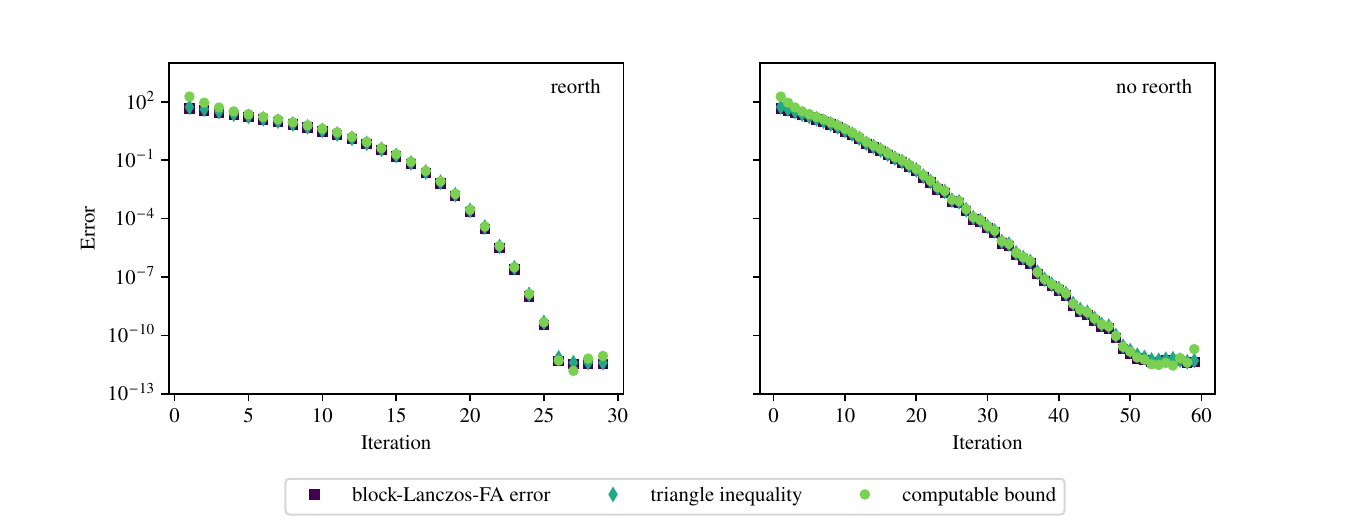}
    \caption{Comparison between the rate of convergence of the block-Lanczos algorithm with and without reorganization for $f(x) = \sqrt{x}$ with block size $b = 4$.
    The contour used is the Pac-Man contour centered at ${\lambda_{\textup{min}}}/{100}$ with $R = 2$ and $\Theta = \frac{1}{100}\pi$.
    The block-Lanczos-FA error is measured with $w = {\lambda_{\textup{min}}}/{100}$. The bound ``triangle inequality'' is \cref{eqn:triangle_ineq} and the bound ``computable bound'' is \cref{thm:main} evaluated numerically.}
    \label{fig:RvT_fp}
\end{figure}

Note that the convergence of the block-Lanczos-FA algorithm is delayed when reorthogonalization is not used. However, the algorithm still converges.
In both cases, our bound performs well, tracking the actual convergence of the algorithm.
This aligns with our expectations from the rounding error analysis performed in \cref{sec:analysis_fp} and provides further evidence supporting the potential value of our bound as a practical stopping criterion.

\section{Conclusion and outlook}

In this paper, we gave an error bound for block-Lanczos-FA used to approximate $f(\vec H)\vec V$ with piece-wise analytic function $f$. 
We analyzed the behavior of our error bound and established that block size does not qualitatively alter the behavior of our error bound; the bounds remain reasonably sharp in all cases.
This allows us to evaluate the behavior of our error bound with a fixed block size, and generalize our findings to other block sizes.
To better understand the impact of contours on our error bound, we investigated the family of the Pac-Man contours and studied its parameters' impact on the slackness of our bound based on our experiments. 
This offered us insights into the qualitative behavior of the contour near the spectrum of $\vec{H}$ and allowed us to choose parameters of the Pac-Man contour that tighten our bound.

In the future, we have two main goals. We will analyze the performance of different combinations of values of $w$ and contour, and look into various other functions, including the sign function, to study the performance of the error bound on a wider range of examples. 

\section*{Declarations}

\paragraph{Ethical Approval}
Not applicable.

\paragraph{Availability of supporting data}
All data generated or analyzed during this study are included in this published article and published on \href{https://github.com/PeterXQC/block_lanczos_CIF}{GitHub}.

\paragraph{Competing interests}
The authors have no competing interests to declare that are relevant to the content of this article.

\paragraph{Authors' contributions}
Analysis and numerical experiments were primarily done by Q.X.. Both authors contributed to writing and reviewing the manuscript.

\bibliography{refs}
\bibliographystyle{siam}
\clearpage

\begin{appendices}
\renewcommand{\thesection}{\arabic{section}.}

\section{Estimates for the block-CG error norm}
\label{sec:cg_est}

In this section, we discuss how to obtain an error estimate for the $(\vec{H}-w\vec{I})$-norm of the error of the block-Lanczos algorithm used to approximate $(\vec{H}-w\vec{I})^{-1}\vec{V}$; i.e. with $f(x)=(x-w)^{-1}$.
For simplicity, without loss of generality,  we will assume $w=0$.
Closely related approaches are widely studied for the block size one case \cite{strakos_tichy_02,meurant_tichy_18,estrin_orban_saunders_19,meurant_papez_tichy_21,meurant_tichy_22}.
Bounds for the block algorithm have also been studied \cite{schaerer_szyld_torrest_21}.

By a simple triangle inequality, we obtain a bound
\begin{align*}
    \| \err_k(w) \|_{\vec{H}} 
    &=  \| \vec{H}^{-1} \vec{V} - \lan_k(f) \|_{\vec{H}}
    \\&\leq \| \vec{H}^{-1}\vec{V} - \lan_{k+d}(f) \|_{\vec{H}} + \| \lan_{k+d}(f) - \lan_k(f)\|_{\vec{H}} .
\end{align*}
If we assume that $d$ is chosen sufficiently large so that 
\begin{equation*}
    \| \vec{H}^{-1}\vec{V} - \lan_{k+d}(f) \|_{\vec{H}} 
    \ll \| \lan_{k+d}(f) - \lan_k(f)\|_{\vec{H}},
\end{equation*}
then we obtain an approximate error bound (error estimate)
\begin{equation}
\label{eqn:cg_bound}
    \| \err_k(w) \|_{\vec{H}} 
    \lesssim \| \lan_{k+d}(f) - \lan_k(f)\|_{\vec{H}}.
\end{equation}

It is natural to ask why we don't simply apply this approach to the block-Lanczos-FA algorithm for any function $f$. 
Such an approach was suggested for the Lanczos-FA algorithm in \cite{eshghi_reichel_21}, and in many cases can work well. 
However, if the convergence of the block-Lanczos-FA algorithm is not monotonic, then the critical assumption that $\lan_{k+d}(f)$ is a much better approximation to $f(\vec{H})\vec{V}$ than $\lan_k(f)$ may not hold.
Perhaps more importantly, the decomposition in \cref{thm:main} also allows intuition about how the distribution of eigenvalues of $\vec{H}$ impact the convergence of block-CG to be extended to the block-Lanczos-FA algorithm.

\begin{figure}[hb]
    \centering
    \includegraphics[scale=.6]{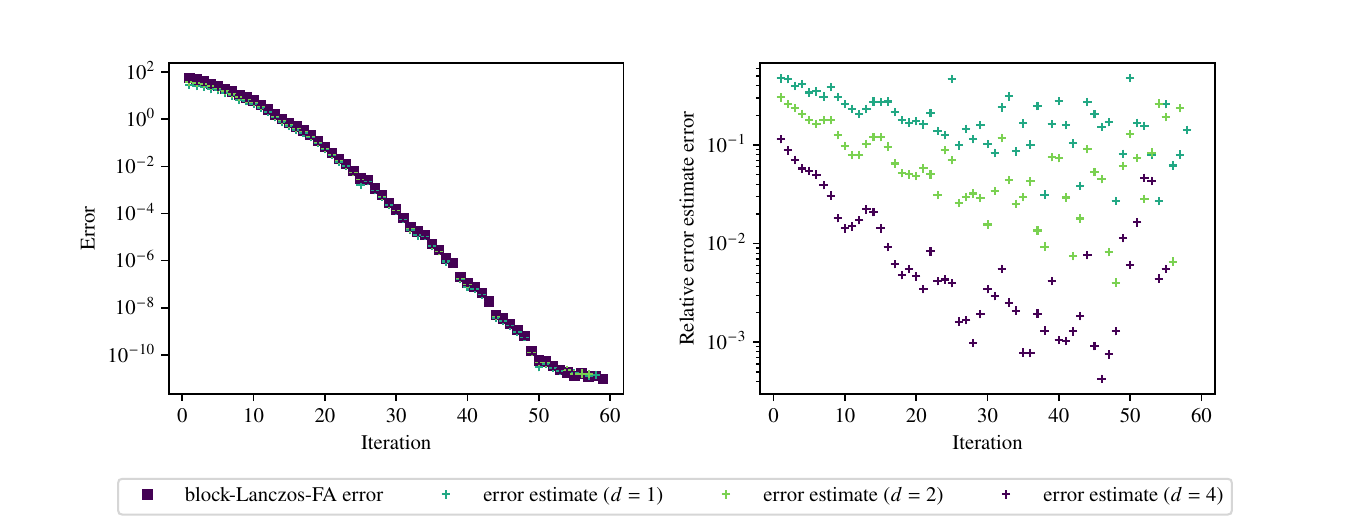}
    \caption{Illustration of the quality of \cref{fig:CG_error_est} for varying $d$.
    The block-Lanczos-FA error is measured using $\|\cdot\|_{\vec{H}-w\vec{I}}$ with $w = 0$.}
    \label{fig:CG_error_est}
\end{figure}

In \cref{fig:CG_error_est} we show a numerical experiment illustrating the quality of \cref{eqn:cg_bound} for several values of $d$.
As expected, increasing $d$ improves the quality of the estimate. 
However, even for small $d$, \cref{fig:CG_error_est} appears suitable for use as a practical stopping criterion on this problem. 
Of course, on problems where the convergence of block-CG stagnates, such an error estimate will require $d$ fairly large.
Further study of this topic is outside the scope of the present paper.

\section{Proofs of known results}
\label{sec:appendix}

\begin{proof}[Proof of \cref{eqn:unif_bound}]

Owing to linearity, in order to show $p(\vec{H})\vec{V} = \lan_k(p)$ for any polynomial $p$ with $\deg(p)<k$, it suffices to consider the case $p(x) = x^j$, for $j=0, \ldots, k-1$. 
Left multiplying \cref{eqn:krylov_recurrence} by $\vec{H}^{j-1}$ and then repeatedly applying \cref{eqn:krylov_recurrence} we find
\begin{align*}
    \vec{H}^j \vec{Q}_k 
    &= \vec{H}^{j-1} \vec{Q}_k \vec{T}_k + \vec{H}^{j-1} \overline{\vec{Q}}_{k+1} \vec{B}_k \vec{E}_k^\cT
    \\&= \vec{H}^{j-2} (\vec{Q}_k \vec{T}_k + \overline{\vec{Q}}_{k+1} \vec{B}_k \vec{E}_k^\cT) \vec{T}_k + \vec{H}^{j-1} \overline{\vec{Q}}_{k+1} \vec{B}_k \vec{E}_k^\cT
    \\&\vdotswithin{=} 
    \\&= \vec{Q}_k \vec{T}_k^{j} + \sum_{i=1}^{j} \vec{H}^{j-i} \overline{\vec{Q}}_{k+1} \vec{B}_k \vec{E}_k^* \vec{T}_k^{i-1}.
\end{align*}
Thus, using that $\vec{V} = \vec{Q}_k \vec{E}_1\vec{B}_0$,
\begin{equation*}
    \vec{H}^j \vec{V} 
    = \vec{H}^j \vec{Q}_k \vec{E}_1\vec{B}_0
    = \vec{Q}_k \vec{T}_k^{j}\vec{E}_1\vec{B}_0 + \sum_{i=1}^{j} \vec{H}^{j-i} \overline{\vec{Q}}_{k+1} \vec{B}_k \vec{E}_k^* \vec{T}_k^{i-1}\vec{E}_1\vec{B}_0.
\end{equation*}
Since $\vec{T}_k$ is banded with half bandwidths $b$, $\vec{T}_k^{i-1}$ is banded with half bandwidth $(i-1)b$.
Thus, the bottom left $b\times b$ block of $\vec{T}_k^{i-1}$ is all zero provided $i-1 \leq k-2$; i.e. $\vec{E}_k^* \vec{T}_k^{i-1}\vec{E}_1\vec{B}_0 = \vec{0}$ provided $i<k$.
We therefore find that for all $j<k$,
\begin{equation*}
    \vec{H}^j \vec{V} 
    =\vec{Q}_k \vec{T}_k^{j}\vec{E}_1\vec{B}_0,
\end{equation*}
from which we infer $p(\vec{H})\vec{V} = \lan_k(p)$ for any $p$ with $\deg(p)<k$.
This is a well-known result; \cite{druskin_knizhnerman_89,saad_92,frommer_lund_szyld_17,frommer_lund_szyld_20}.

Therefore, using the triangle inequality, for any polynomial $p$ with $\deg(p)<k$,
\begin{align*}
    \| f(\vec{H}) \vec{V} - \lan_k(f) \|_2 
    &= \| f(\vec{H}) \vec{V} - \vec{Q}_kf(\vec{T}_k) \vec{E}_1 \vec{B}_0 \|_2
    \\&\leq \| f(\vec{H}) \vec{V} - p(\vec{H}) \vec{V} \| + \| \vec{Q}_k p(\vec{T}_k) \vec{E}_1 \vec{B}_0 - \vec{Q}f(\vec{T}_k) \vec{E}_1 \vec{B}_0 \|_2
    \\&\leq \| f(\vec{H}) - p(\vec{H}) \|_2 \| \vec{V} \|_2 + \| \vec{Q}_k \|_2 \| p(\vec{T}_k) - f(\vec{T}_k) \| \|\vec{E}_1 \vec{B}_0 \|_2.
    \\&\leq \| \vec{V} \|_2 \left( \| f(\vec{H}) - p(\vec{H}) \|_2 +  \| p(\vec{T}_k) - f(\vec{T}_k) \|_2 \right).
\end{align*}
In the final inequality we have use that $\|\vec{Q}_k\|\leq 1$ since $\vec{Q}_k$ has orthonormal columns and that $\|\vec{E}_1 \vec{B}_0\| = \|\vec{Q}_k^\cT \vec{V}\| \leq  \|\vec{V}\|$.
Next, note that
\begin{equation*}
    \| f(\vec{H}) - p(\vec{H}) \|_2
    = \max_{x\in\Lambda(\vec{H})} | f(x) - p(x) |
    ,\qquad
    \| p(\vec{T}_k) - f(\vec{T}_k) \|_2
    = \max_{x\in\Lambda(\vec{T}_k)} | f(x) - p(x) |
\end{equation*}
so, since each of $\Lambda(\vec{H})$ and $\Lambda(\vec{T}_k)$ are contained in $[\lambda_{\textup{min}}(\vec{H}),\lambda_{\textup{max}}(\vec{H})]$, we find 
\begin{align*}
    \| f(\vec{H}) \vec{V} - \lan_k(f) \|_2
    &\leq 2 \| \vec{V} \| \max_{x\in[\lambda_{\textup{min}}(\vec{H}),\lambda_{\textup{max}}(\vec{H})]} | f(x) - p(x) |.
\end{align*}
Since $p$ was arbitrary, we can optimize over all polynomials $p$ with $\deg(p)<k$.    
\end{proof}

\begin{proof}[Proof of \cref{thm:res_prop_Qk1}]
From \cref{def:err_and_res}, 
\begin{equation*}
    \res_k(z) := \vec{V} - (\vec{H}-z\vec{I}) \vec{Q}_k (\vec{T}_k - z\vec{I})^{-1}  \vec{E}_1\vec{B}_0.
\end{equation*}
Using the fact that the Lanczos factorization \cref{eqn:krylov_recurrence} can be shifted, 
\begin{equation*}
    (\vec{H}-z\vec{I}) \vec{Q}_k = \vec{Q}_k(\vec{T}_k-z\vec{I})+\overline{\vec{Q}}_{k+1}\vec{B}_k\vec{E}_k^\cT.
\end{equation*}
Thus, by substitution, we have 
\begin{align*}
    \res_k(z) &= \vec{V} - (\vec{Q}_k(\vec{T}_k-z\vec{I})+\overline{\vec{Q}}_{k+1}\vec{B}_k\vec{E}_k^\cT) (\vec{T}_k - z\vec{I})^{-1}  \vec{E}_1\vec{B}_0\\
    &= \vec{V} - \vec{Q}_k\vec{E}_1\vec{B}_0-\overline{\vec{Q}}_{k+1}\vec{B}_k\vec{E}_k^\cT(\vec{T}_k - z\vec{I})^{-1} \vec{E}_1\vec{B}_0
\end{align*}
Since $\vec{E}_1 = \left[ \vec I_{b}, \vec 0, ..., \vec 0 \right]^\cT$, where $\vec I_{b}$ is the $b\times b$ identity matrix, we know that $\vec Q\vec E_1\vec{B}_0 = \vec Q_1\vec{B}_0 = \vec V$. 
Therefore,
\begin{align*}
    \res_k(z) &= \vec{V} - \vec{Q}_k\vec{E}_1\vec{B}_0-\overline{\vec{Q}}_{k+1}\vec{B}_k\vec{E}_k^\cT(\vec{T}_k - z\vec{I})^{-1}  \vec{E}_1\vec{B}_0\\
    &= -\overline{\vec{Q}}_{k+1}\vec{B}_k\vec{E}_k^\cT(\vec{T}_k - z\vec{I})^{-1}  \vec{E}_1\vec{B}_0.
\end{align*}
The result follows from the definition of $\vec{C}_k(z)$.
\end{proof}

\begin{proof}[Proof of \cref{thm:err_res_shift}]
We have shown in \cref{thm:res_prop_Qk1} that $\res_k(z) = \overline{\vec{Q}}_{k+1}\vec{B}_k \vec{C}_k(z)$.
Thus,  
\begin{equation*}
    \res_k(z) = \overline{\vec{Q}}_{k+1}\vec{B}_k \vec{C}_k(z) = \overline{\vec{Q}}_{k+1}\vec{B}_k \vec{C}_k(w)\vec{C}_k(w)^{-1} \vec{C}_k(z) = \res_k(w) \vec{C}_k(w)^{-1} \vec{C}_k(z).
\end{equation*}
Next, notice that 
\begin{align*}
\err_k(z) = (\vec{H}-z\vec{I})^{-1}\res_k(z)
    &=(\vec{H}-z\vec{I})^{-1}\res_k(w)\vec{C}_k(w)^{-1} \vec{C}_k(z)\\
    &= h_{w, z}(\vec{H})(\vec{H}-w\vec{I})^{-1}\res_k(w)\vec{C}_k(w)^{-1} \vec{C}_k(z)\\
    &=h_{w, z}(\vec{H})\err_k(w)\vec{C}_k(w)^{-1} \vec{C}_k(z). 
\end{align*}
The result is established.
\end{proof}
\end{appendices}
\end{document}